\documentclass[12pt,toc,page,title]{article} 

\usepackage{setspace}
\usepackage{PRIMEarxiv}

\usepackage[utf8]{inputenc} % allow utf-8 input
\usepackage[T1]{fontenc}    % use 8-bit T1 fonts
\usepackage{hyperref}       % hyperlinks
\usepackage{url}            % simple URL typesetting
\usepackage{booktabs}       % professional-quality tables
\usepackage{amsfonts}       % blackboard math symbols
\usepackage{nicefrac}       % compact symbols for 1/2, etc.
\usepackage{microtype}      % microtypography
\usepackage{lipsum}
\usepackage{fancyhdr}       % header
\usepackage{graphicx}       % graphics
\graphicspath{{media/}}     % organize your images and other figures under media/ folder

\usepackage{graphicx}
\usepackage{subcaption}
\usepackage{mathptmx}      % use Times fonts if available on your TeX system
\usepackage[english]{babel}
\usepackage{appendix}
\usepackage{enumitem}
\usepackage{url}            % simple URL typesetting
\usepackage{booktabs}       % professional-quality tables
\usepackage{amsfonts}       % blackboard math symbols
\usepackage{nicefrac}       % compact symbols for 1/2, etc.
\usepackage{microtype}      % microtypography
\usepackage{lipsum}
\usepackage{amsthm}
\usepackage{amssymb}
\usepackage{amsmath}
\usepackage{mathptmx}
\usepackage{hyperref}       % hyperlinks
\usepackage{xcolor}
\usepackage{comment}
\usepackage{soul}
\usepackage{comment}

\usepackage{algorithm}
\usepackage[noend]{algpseudocode}
\makeatletter
\def\BState{\State\hskip-\ALG@thistlm}
\makeatother

\usepackage[english]{babel}
\DeclareMathOperator*{\esssup}{ess\,sup}

\DeclareMathOperator*{\argmin}{arg\,min}

\newcommand{\dd}{\mathrm{d}}

\newcommand{\1}{\mathbf{1}}

\newcommand{\ie}{\emph{i.e.}}

 \newtheorem{theorem}{Theorem}
 \newtheorem{lemma}{Lemma}
 \newtheorem{corollary}{Corollary}

 \newtheorem{definition}{Definition}

\usepackage[sort&compress]{natbib}
 \bibpunct[, ]{[}{]}{,}{n}{}{,}%
\title{Bayesian Risk Preference Persuasion
}

\author{
  Shutian Liu \\
  Department of Systems Engineering\\
  City University of Hong Kong\\
  \texttt{shutian.liu@cityu.edu.hk} \\
}

\begin{document}
\maketitle

\begin{abstract}
A decision-maker’s risk preference is inherently unstable and may adjust in response to external information, shaping subsequent choices and their outcomes. This paper develops a persuasion framework to study how information can be designed to steer risk preferences and decision results. In our framework, a receiver starts with an initial risk preference represented by a coherent risk measure and revises it after observing a system state generated by an information rule claimed by a sender. The revision must preserve time consistency of risk evaluations before and after the state realization. We investigate two models under this framework: one focuses purely on how receiver's risk preference is influenced by information, the other also incorporates receiver's decision. For both models, we investigate the sender’s optimal information design problem by analyzing the induced distribution of posterior beliefs over states. Each belief leads to specific preference revisions and corresponding conditional risk assessments. We identify conditions under which information design benefits the sender across several settings and illustrate the framework’s potential in risk management through an application to reinsurance design.
\end{abstract}

% keywords can be removed

\section{Introduction}
\label{sec:intro}

The risk preference of a decision-maker (DM) is a foundational variable in risk management frameworks. For instance, when evaluating the theoretical fair value of a market asset, assuming a risk-neutral perspective is essential for objective pricing. Conversely, in safety-critical scenarios, DMs must adopt a strictly risk-averse posture. In these high-stakes environments, a robust, failure-resistant strategy is optimal, even if its conservative nature sacrifices substantial expected returns. Consequently, a key challenge of a risk manager is ensuring that executives possess and apply the appropriate risk preferences to match their specific operational scenarios.
This raises a critical question: is there a systematic way to steer a DM's risk preference so that their choices better align with the specific decision environment?

While the majority of the literature on decision-making under uncertainty adopt the setting in which the risk preferences are predetermined, evidences have indicated that risk preference lacks stability (see, e.g., \citet{tversky1990causes, berg2005risk, handel2015health}). 
Exogenous inputs, such as information, can fundamentally shift individual risk preferences and drive subsequent behavioral changes (see, e.g., \citet{barseghyan2011risk, nasioulas2026feedback, gandhi2017information}). With the rapid advancement of information technology and artificial intelligence, systematically decoding how risk preferences depend on information flows is increasingly vital. Uncovering these mechanisms is crucial to mitigating unintended behavioral volatility and integrating predictive behavioral insights directly into modern risk management frameworks.

Motivated by these observations, the main objective of this paper is to formally investigate how to design persuasive information to steer a DM's risk preference and shape the resulting behavior under uncertainty.
The main contributions of this paper are summarized as follows.
\begin{itemize}
    \item \textbf{Modeling framework}. We propose a new approach to study the design of information to shape the risk preference of a decision-maker. Our framework builds on the Bayesian persuasion model and the time-consistent decomposition of risk measures but deviates from conventional considerations. On the one hand, unlike the majority of previous works in Bayesian persuasion and information design, we explicitly focus on the influence of information on the risk preference of the receiver.
    On the other hand, instead of purely focusing on the analysis of the underlying axioms and properties of risk measures, we take the perspective of persuasive information construction to induce preferred preference revisions at a later decision stage. 
    To facilitate the investigation of preference persuasion, we introduce a system state parameter that serves as the intermediate variable linking the persuasive information signal and the uncertainty being evaluated.
    In particular, we adopt the setting where the joint probability distribution of the system state and the underlying uncertainty is incompletely known ex ante.
    Instead, we assume that it has to be constructed using the conditional probability distribution over the uncertainty at each system state and the posterior beliefs about system states induced by the persuasive signals.
    This feature highlights the systematic structure of uncertainty quantification where risk assessment is only feasible and meaningful conditional on a specific realization of the system state and knowledge about the probabilistic rule of the system state is captured by one's belief.
    With this model construction, we are able to describe how the information rule designed and claimed by the sender could shape the receiver's beliefs about the system states, hence illustrate how it induces the intended interim revisions of the receiver's risk evaluations. 
    Depending on whether the action of the receiver is explicitly incorporated into the sender's persuasion problem, we formulate preference persuasion per se and preference persuasion with actions.
    The first class of problems emphasizes how information design would steer risk preferences in the target direction, while the second class of problems is used to evaluate the end-to-end effect of preference persuasion. 
    \item \textbf{Theoretical analysis}. Our analysis of the sender's persuasive information design problem follows the strand of research but tailored to the context of preference persuasion. 
    For both the two classes of the sender's problem, we are interested in the questions of whether there exists a signal rule that achieves the sender's optimal design problem and when the sender can benefit from persuasion.
    For preference persuasion per se, in particular, we render the dual variable associated with risk measures that induces preference revisions as the ``action" of the receiver in the absence of her true action.
    This necessarily identifies the ``action" set as a continuum.
    We view this set as a parameterized constraint system and investigates its hemicontinuity properties.
    To address the second question, we develop a technique to bridge the transportation distance between marginal probability distributions and the transportation distance between those joint probability distributions that can be disintegrated in to these marginals and the same stochastic kernel. 
    Combining with the standard assumption of ``there is information that the sender would share", this technique helps derive the conditions under which the sender can increase or decrease the average risk associated with the receiver’s persuaded preferences. Thus, it determines when persuasion is beneficial to her.
    Our analysis differs from the literature where the benefit of sender is primarily analyzed under the setting of a finite set of the receiver's actions. 
    For preference persuasion with actions, we simplify the structure of the sender's problem by assuming that the optimal dual variables inducing preference revisions are unique.
    This simplification enables us to consider the preference adaptation and decision-making of the receiver as a whole and to follow standard techniques for analysis. 
    \item \textbf{Application}. Among the various applications compatible with our theoretical framework, we use reinsurance design as an illustrative example to highlight the potential of our method in risk management.
    To this end, we first extend the commonly used optimal reinsurance design model under average value-at-risk (AV@R) risk measures to its state-dependent counterpart. In this extension, optimal indemnity function is characterized under specific realizations of the system state based on conditional evaluation of risks. 
    We then grant the reinsurer the capability to design information about the states, making him the sender in our preference persuasion framework. 
    We investigate several parameter settings and identify critical distributions of persuaded beliefs to elaborate on when persuasive information may or may not assist reinsurance design. 
    Under the scenarios where the reinsurer is better off if he is a sender, we compute the optimal expected losses he can achieve and construct the corresponding signal rules that attain these values. 
    The adoption of preference persuasion enriches traditional reinsurance problems by introducing an additional degree of freedom of design to the reinsurer.
    This perspective has potential in further enhancing the risk-sharing mechanism and improving the resilience of the insurer-reinsurer contractual relationship.
\end{itemize}

%\subsection{Related works}
%\label{sec:related works}
\subsection{Literature review}
Our study is related to two broad literatures: information design and risk-averse decision-making.

Following the theoretical foundations of \citet{kamenica2011bayesian} and \citet{bergemann2016bayes}, Bayesian persuasion and information design have found various application scenarios in operations and management.
For example, \citet{candogan2025value} study information design from the perspective of a retailer in a supply chain, \citet{maitra2024optimal} consider optimal signaling for infection risk mitigation, and \citet{anunrojwong2024persuading} investigate how persuasive signals improve services in unobservable queues.
In particular, \citet{anunrojwong2024persuading} adopt the modeling choice of risk conscious receivers whose utility functions may be nonlinear in beliefs.
This setting leads to the failure of revelation principal and an alternative method is considered for signal rule construction. 
Related considerations that also deviate from the risk neutral setting include  \citet{lipnowski2018disclosure}, \citet{beauchene2019ambiguous}, and \citet{babichenko2022information}.
Different from these existing works, we take the perspective of leveraging the instability of DM's risk preference and focus on how information design could induce preference shifts and result in intended behavioral outcomes. 
This adaptive description of the DM's risk preference has recently been investigated in \citet{liu2026games}, where the author has focused on the comparison of different equilibrium notions in incomplete information games but the information structure itself is fixed. 
In addition, \citet{cabrales2013entropy} study the value of purchasing of information to improve decision under uncertainty.
They focus on the receiver who proactively obtains information while we describe the incentives on the sender's side.

Our work also draws insights from the literature on decision-making under uncertainty, especially works in which a predefined risk preference of the DM is absent.
A non-exhaustive list of related work includes: \citet{wang2023preference} who extend the Anscombe–Aumann framework (\citet{anscombe1963definition}) and consider the weighted average of distortion risk measures at each state where the weights are determined by subjective probabilities; \citet{armbruster2015decision}, \citet{guo2022robust}, and
\citet{delage2022shortfall} who consider incomplete knowledge about the subjective risk preferences of DMs; and \citet{zhu2009worst} and
\citet{li2025randomization} who adopt random preferences for DMs due to distributional uncertainty. 
The recent work of \cite{tao2026risk} adopts closely related descriptions of DM's preferences under incomplete information, where nested constructions based on conditional risks and Bayesian predictive distributions are used to evaluate uncertainties in distinguishable and unified perspectives.
Our work is also related to \citet{liu2025stackelberg} and \citet{liu2025mitigating}, where the authors have introduced the perspective of designing risk preferences in Stackelberg game models. However, they have made the assumptions that the risk preference of a DM always follows a designer's choice instead of being derived consistently based on information.

\subsection{Paper organization}
%\label{sec:paper organization}
The rest of the paper is organized as follows.
In Section \ref{sec:framework}, we present the Bayesian risk preference persuasion framework by introducing the state and signal construction, the preference revision feature, and the sender's persuasion problems.
In Section \ref{sec:analysis:per se}, we ignore the action of the receiver and investigate the limiting bound of the effect of persuasive information and the sufficient conditions for persuasion to shape the average risk preference of the receiver in targeted directions. 
A numerical example is also presented to illustrate the detailed procedure to perform preference persuasion per se. 
Analysis incorporating the receiver's action will be conducted in Section \ref{sec:analysis with action}.
Section \ref{sec:application} investigate the reinsurance application in detail to provide insights on how the proposed framework can assist risk management.
Finally, conclusion and discussions on future works are included in Section \ref{sec:conclusion}.

\section{Framework}
\label{sec:framework}

\subsection{Uncertainty, states, and signals}
\label{sec:states}
Consider a sender and a receiver interacting under uncertainty $\omega\in\Omega$.
To make the presentation centered around risks, we assume that both the sender and the receiver minimize losses instead of maximize gains.
Consider a finite state space $T$, whose realizations $t\in T$ influence the perceived probabilistic rules of the uncertainty.
Let $\Delta(\Omega)$ denotes the set of probability measures on $T$.
For each state $t$, consider the probability measure $P(\cdot|t)\in \Delta(\Omega)$ that encodes the distributional information of the uncertainty under that state. 
Assume that $P(\cdot|t)$ for all $t\in T$ is commonly known by both the sender and the receiver. 
We will also write $P_t$ for $P(\cdot|t)$ for notational simplicity.

Depending on whether a specific state $t$ is revealed or not, two stages are in place. 
The ex ante stage is where $t$ has not been revealed and information about states are summarized by a probability distribution $\mu\in \Delta(T)$, called a belief. 
The prior belief of the receiver at the beginning of the interaction is denoted $\mu_0\in \Delta(T)$, which is assumed to be commonly known.
Given $\mu$ and $P(\cdot|t)$ for all $t\in T$, a mixture distribution $Q:=\mu\circ P \in \Delta(\Omega)$ defined by
\begin{equation*}
\mu\circ P:=  \sum_{t\in T} P(\cdot|t)\mu(t)
\end{equation*}
can be used to describe information about the uncertainty $\omega$.
The interim stage is where a specific state $t$ is observed and one can resort to $P(\cdot|t)$ for probabilistic information about the uncertainty $\omega$.

While a state observation is only made later at the interim stage, a private signal $s\in S$ is available ex ante that encodes information about the states. 
Let $\pi(\cdot|t)$ denote a signaling rule such that $\pi(\cdot|t)\in \Delta(S)$ for all $t\in T$.
The signaling rule $\pi$, together with a signal observation $s$ selected based on the rule, are assumed to be known by the receiver ex ante. 
This information structure leads to the following belief update rule based on Bayes theorem.
For signal realization $s$, let $\mu_s\in\Delta(T)$ denote the posterior belief derived based on the prior belief $\mu_0$ and the signaling rule $\pi$ according to 
\begin{equation}
    \mu_s(t):=\frac{\pi(s|t)\mu_0(t)}{\sum_{t'\in T}\pi(s|t')\mu_0(t')}, \forall s\in S, \forall t\in T.
    \label{eq:Bayes rule}
\end{equation}
Then, each signaling rule $\pi$ leads to a distribution over posterior beliefs $\eta\in \Delta(\Delta(T))$ defined as 
\begin{equation*}
    \eta(\mu):=\sum_{s:\mu_s=\mu}\sum_{t'\in T}\pi(s|t')\mu_0(t'), \forall \mu\in \Delta(T),
\end{equation*}
where the support of $\eta$ is denoted $\text{supp}(\eta)=\{\mu_s\}_{s\in S}$.
Throughout this paper, we assume that the sender and the receiver share the same belief.

A distribution of posterior beliefs is Bayes-plausible (see \citet{kamenica2011bayesian}) if the expectation of the posteriors with respect to that distribution equals the prior, \ie,
\begin{equation*}
    \sum_{\text{supp}(\eta)}\mu\eta(\mu)=\mu_0,
\end{equation*}
where $\text{supp}(\eta)$ denotes the support of $\eta$.

In the above specifications of states and signals, the parameter that directly influences the receiver's actions and losses is the uncertainty $\omega$.
This means that the persuasive effort of the sender, endowed in the signal rule $\pi$, indirectly affects the receiver via the information about the state $t$ it encodes. 
As the realization of a state determines the stage of decision-making, the corresponding criterion for the quantification of the uncertainty can be either unconditional and defined with respect to the mixture distribution, or conditional and defined with respect to a distribution under a specific state realization.
This setting builds on and extends standard specifications of information structures.

Next, we describe how the receiver evaluates the risks arising from uncertainty.

\subsection{Risk preference and its revision}
\label{sec:risk preferences}
In standard persuasion or information design models, both the sender and the receiver are assumed to be rational.
They assess random losses in a risk-neutral perspective.
We maintain the assumption that the sender is risk-neutral and he uses expectation to quantify the uncertainty in the loss function.
However, we introduce risk aversion to the receiver and assume that she enters the interaction with an initial risk preference represented by a law-invariant coherent risk measure (CRM) defined as follows.

We assume that the receiver's random loss $Y$ is contained in the space $\mathcal{L}^{\infty}(\Omega,\mathcal{F},Q)$ of all essentially bounded, $\mathbb{R}$-valued random variables on the probability space $(\Omega,\mathcal{F},Q)$. 
As illustrated by the notation, we will later identify $Q$ using the mixture distribution defined in Section \ref{sec:states}.
For now, we refer to $Q$ as a given reference probability measure associated with space $(\Omega, \mathcal{F})$. 
A risk functional $\rho$ is a mapping from  $\mathcal{L}^{\infty}(\Omega,\mathcal{F},Q)$ to $\mathbb{R}$.
The following axiomatic definition delineates the class of risk functionals admissible within the model considered in this paper.
We will subsequently specialize to a particular specification to further simplify our presentation.
See, for example, \citet{artzner1999coherent, follmer2025stochastic, ruszczynski2006optimization} for further discussions of widely used classes of risk measures and their properties.

\begin{definition}
\label{def:CRM}
(Law-invariant coherent risk measure).
A risk functional $\rho:\mathcal{L}^{\infty}(\Omega,\mathcal{F},Q)\rightarrow \mathbb{R}$ is a law-invariant CRM if it satisfies the following properties:
\\
(i) Monotonicity: $\rho(Y_1)\leq \rho(Y_2)$ if $Y_1\leq Y_2$ a.s.; 
\\
(ii) Convexity: $\rho ((1-\alpha)Y_1+\alpha Y_2)\leq (1-\alpha)\rho(Y_1)+\alpha \rho ( Y_2)$ for $0\leq \alpha \leq 1$;
\\
(iii) Translation invariance: $\rho(Y+c)=\rho(Y)+c$ for $c\in\mathbb{R}$;
\\
(iv) Positive homogeneity: $\rho(\alpha Y)=\alpha \rho(Y)$ for $\alpha>0$.
\\
(v) Law-invariance:
$\rho(Y_1)=\rho(Y_2)$ whenever $P(Y_1\leq y)=P(Y_2\leq y)$ for all $y\in\mathbb{R}$.
\end{definition}

\paragraph{Information-contingent preference revision.}

Unlike standard models in which risk preferences are assumed to remain fixed throughout the decision-making process, we allow for the possibility that preferences may be revised.
Our modeling assumption is as follows. 
The receiver revises her risk preference upon observing the state realization $t\in T$.
The revision is not arbitrary but is performed to ensure the time-consistency of risk evaluations across the ex ante and the interim stages.
The motivation of this requirement stems from the observation that preferences are not stable and the consideration that DMs are not myopic.
This setting is enabled with the help of the extended conditional risk functional and the decomposition of risk measures introduced in \citet{pflug2016time}.
The definitions are presented  following the introduction of the dual representation of CRM useful for the constructions.

A coherent risk measure $\rho:\mathcal{L}^{\infty}(\Omega,\mathcal{F},Q)\rightarrow \mathbb{R}$ admits the following representation via the Fenchel-Moreau duality theorem (see, e.g., \citet{artzner1999coherent,ruszczynski2006optimization,follmer2025stochastic}):
\begin{equation}
    \rho(Y)=\sup \{ \mathbb{E}(YZ): Z\in\mathfrak{M}\subset\mathcal{L}^1(\Omega,\mathcal{F},Q) \},
    \label{eq:dual representation risk measure}
\end{equation}
where $\mathfrak{M}$ is the dual set consisting of density functions $Z(\cdot)$ that are absolutely continuous with respect to the reference probability measure $Q$ and satisfy that $Z$ is nonnegative , $Z(\Omega)=1$, and $\mathbb{E}(YZ)\leq \rho(Y)$ for all $ Y\in \mathcal{L}^{\infty}(\Omega,\mathcal{F},Q)$.

The extended conditional risk functional is defined with respect to a sub-sigma algebra $\mathcal{F}_\tau\subset\mathcal{F}$ as follows.
\begin{definition}
\label{def:extended conditional risk functionals}
(Extended conditional risk functional).
Let $\rho$ be a law-invariant coherent risk measure. 
For dual variables $Z_\tau$ measurable with respect to $\mathcal{F}_\tau$ that satisfies $Z_\tau\geq 0$, $\mathbb{E}(Z_\tau)=1$, and $\mathbb{E}(YZ_\tau)\leq \rho(Y)$ for all $Y\in \mathcal{L}^{\infty}(\Omega,\mathcal{F},Q)$, the extended conditional risk functional associated with $\rho$ is defined as
\begin{equation}
    \rho_{Z_\tau}(Y|\mathcal{F}_\tau):=
    \esssup\{\mathbb{E}(YZ')| Z'\in \mathfrak{M}_{Z_\tau}\subset \mathcal{L}^1(\Omega,\mathcal{F},Q) \},
    \label{eq:extended conditional risk functional definition}
\end{equation}
where $\mathfrak{M}_{Z_\tau}$ denotes the set of dual variables associated with $\rho_{Z_\tau}(L|\mathcal{F}_\tau)$ defined as
\begin{equation}
    \mathfrak{M}_{Z_\tau}:=
     \{Z':\mathbb{E}(Z'|\mathcal{F}_\tau)=\1,
Z'\geq 0,
\text{and } \mathbb{E}(YZ_\tau Z')\leq \rho(Y)   
\},
\end{equation}
for all $Y\in\mathcal{L}^{\infty}(\Omega,\mathcal{F},Q)$.
\end{definition}
The above definition has the following interpretation.
The sub-sigma algebra $\mathcal{F}_\tau$ represents additional information for uncertainty quantification that is not available when the evaluation is with respect to the sigma algebra $\mathcal{F}$.
Thus, the risk measure $\rho$ evaluates risks under incomplete information about the randomness and $\rho_{Z_\tau}$ serves as its natural extension when observations that provide information in $\mathcal{F}_\tau$ are made.

A consequence of extended conditional risk functionals lies in the following decomposition theorem, which enables time-consistent risk evaluations before and after obtaining additional information.
\begin{lemma}
\label{lemma:decomposition}
(Theorem 21 of \citet{pflug2016time}).
Let $\rho$ denote a law-invariant coherent risk measure. Then, the following holds
\begin{equation}
    \rho(Y)=\sup\mathbb{E}[Z_\tau \cdot \rho_{Z_\tau}(Y|\mathcal{F}_\tau)],
    \label{eq:decomposition of risk measure}
\end{equation}
where the supremum is among all $\mathcal{F}_\tau$-measurable dual variables $Z_\tau$ satisfying $Z_\tau\geq 0$, $\mathbb{E}(Z_\tau)=1$, and $\mathbb{E}(YZ_\tau)\leq \rho(Y)$ for all $Y\in \mathcal{L}^{\infty}(\Omega,\mathcal{F},Q)$.
\end{lemma}

We identify the elements used in defining the extended conditional risk measures and the decomposition using the settings presented in Section \ref{sec:states} to enable risk preference persuasion in Section \ref{sec:preference persuasion model}.
First, we identify a specific realization of the state $t$ with a random variable $\tau$ from $\Omega$ to $T$.
Then, by the assumption that $T$ is finite, $\mathcal{F}_\tau$ is explicitly generated by a finite partition of $\Omega$ for each state $t\in T$.
Hence, we can consider $P(\cdot|t)$ as a conditional probability measure with respect to $\mathcal{F}_\tau$ at a realization $t$.
Then, for any event in $\mathcal{F}$, the conditional probability of this event is an $\mathcal{F}_\tau$-measurable random variable. 
Since $T$ is finite, this probability can be expressed as the mixture of $P(\cdot|t)$ when the realization of $\tau$ is $t$. 
Finally, the probability of this event evaluated with respect to $\mathcal{F}$ will be represented by the mixture $Q$.
We will adhere to the notations depending on the realized state $t$ instead of the sub-sigma algebra in the rest of paper for notational convenience.

\paragraph{Specific forms under AV@R.}
Since the focus of this paper is on the persuasion of risk preferences, in the sequel, we adopt the following specific law-invariant CRM to simplify our presentation. 
The AV@R measure at confidence level $\alpha\in [0,1)$ is defined as
\begin{equation*}
    \text{AV@R}_\alpha(Y):=(1-\alpha)^{-1}\int_\alpha^{1}\text{V@R}_\gamma(Y)d\gamma,
\end{equation*}
where 
\begin{equation*}
    \text{V@R}_\alpha(Y):=\inf\{y:F_Y(y)\geq \alpha\},
\end{equation*}
and $F_Y(\cdot)$ is the cumulative distribution function of $Y$.
At level $\alpha=1$, define $\text{AV@R}_1(Y):=\esssup (Y)$.
The AV@R measure is frequently used in the literature not only for its convenience in optimization but also for its role as a building-block for defining other classes of CRMs (see, e.g., \citet{shapiro2013kusuoka}).

As a law-invariant CRM, AV@R admits the following extended conditional version in terms of Definition \ref{def:extended conditional risk functionals}:
\begin{equation}
    \text{AV@R}_{\alpha,Z_t}(Y|t):=\text{AV@R}_{1-(1-\alpha)Z_t}(Y|t).
    \label{eq:avar at random level}
\end{equation}
The conditional risk measure (\ref{eq:avar at random level}) is referred to as conditional AV@R at random level in \citet{pflug2016time} and its time-consistent decomposition satisfies:
\begin{equation}
    \text{AV@R}_{\alpha}(Y)=\sup \mathbb{E}[Z_t\cdot \text{AV@R}_{1-(1-\alpha)Z_t}(Y|t)],
    \label{eq:decomposition of avar}
\end{equation}
where the supremum is over the set of $\mathcal{F}_\tau$-measurable density functions 
\begin{equation*}
    \mathcal{Z}(\mu):=\{Z_t, t\in T : \mathbb{E}(Z_t)=1, Z_t\geq 0, \text{ and } (1-\alpha)Z_t \leq \1\},
\end{equation*}
for all $\mu\in \Delta(T)$.
Note that the expectation constraint $\mathbb{E}(Z_t)=1$ in the definition of set $\mathcal{Z}(\mu)$ is with respect to the probability measure $\mu \in \Delta(T)$.
This probability measure will be clear from the contest in the rest of the paper as we will either define it explicitly or identify it with a mixture $Q$ that is induced by the considered $\mu$ and the fixed $P(\cdot|t)$.
We will denote by $\mathcal{Z}^*(\mu)$ the set of optimizers of (\ref{eq:decomposition of avar}).

\subsection{Risk preference persuasion}
\label{sec:preference persuasion model}
Having introduced receiver's risk preference and how it is revised contingent on observed information, we now proceed to formulate the Bayesian risk preference persuasion problem, which integrates the signaling effect and the preference revision feature. 
Note that the literature on persuasion and information design has largely focused on settings in which the ultimate objective of the sender is to influence the receiver’s action.
We emphasize that, within our framework, examining persuasion problems that target the receiver’s risk preferences directly is at least as meaningful, if not more so, than the conventional action-centered formulation.
Therefore, we investigate both scenarios in the optimization problems in the following and in their analysis in later sections. 

\paragraph{Preference persuasion per se.}
At the first glance, it may seem invalid to consider a persuasion problem if the action of the receiver is ignored.
However, due to the preference revision feature, the persuasion problem remains valid if the receiver is assumed to control the manner in which preference revisions occur.
In particular, the dual variables $\{Z_t\}_{t\in T}$, which are determined through the optimization problem (\ref{eq:decomposition of avar}) and induce the revised risk preferences $\text{AV@R}_{\alpha, Z_t}$ for $t\in T$, are assumed to be selected by the receiver based on her persuaded posterior beliefs. 
Then, the sender's problem can be formulated as 
\begin{equation}
    \begin{aligned}
        \inf_{\eta, \{Z_t^*\}_{t\in T}} \quad &
        \mathbb{E}_{\eta}\mathbb{E}_{\mu}\text{AV@R}_{\alpha, Z_t^*}(Y|t) \\
        \text{s.t.} \quad &
        \sum_{\text{supp}(\eta)}\mu \eta(\mu)=\mu_0, \\
        & \{Z_t^*\}_{t\in T}\in \mathcal{Z}^*(\mu).
        \label{eq:persuasion without action}
    \end{aligned}
\end{equation}
In (\ref{eq:persuasion without action}), the sender, seeking to minimize the expectation of perceived risks evaluated according to the revised preferences, designs the distribution of the receiver's posterior beliefs $\eta$ such that the induced beliefs $\mu$ generate favorable risk preference revisions, as characterized by $\{Z_t^*\}_{t\in T}$.
The minimization in $\{Z_t^*\}_{t\in T}$ is in line with the consideration of the sender-preferred subgame perfect equilibrium in the literature of information design.
This solution concept allows the sender to select among optimal receiver responses in his favor in case multiplicity occurs.
The sender may also consider maximization in the above problem depending on specific application. 
We will discuss this further in the example in Section \ref{sec:example}.
Our analysis in later sections focuses on how the distribution of beliefs would change the revised preferences on average and does not limit to only minimization. 
Note that we adopt the setting of sender-preferred subgame perfect equilibrium throughout the discussions of the paper, as deviations from this solution concept significantly increase the complexity of the problems in terms of epistemic details. 
For example, first-order beliefs may not be sufficient for analyzing the persuasion problem. We refer the reader to \citet{mathevet2020information} for additional details.

%For a given random loss $Y$, we can define for a given risk preference $\rho(\cdot)$ and a mixture distribution $Q\in \Delta(\Omega)$ the optimal dual variable $Z$ such that $\rho(Y)=\mathbb{E}_Q[YZ]$.

\paragraph{Preference persuasion with actions.}
In this scenario, the receiver is understood as a decision-maker (DM) interested in selecting an action $a$ from a finite set of actions $A$.
The loss functions of the sender and the receiver, as suggested by previous settings, depends on both the receiver's action and the state. 
The sender's continuous loss function is denoted $v(a,t)$. 
The fact that $v$ is independent of the randomness $\omega$ reflects our modeling assumption that the sender has no control over an action that directly responds to the randomness and he is forced to rely on persuading the receiver on this matter. 
The receiver's action-dependent random loss is $X_a(\omega)$. 
Depending on whether the random loss is evaluated ex ante or at the interim stage, the resulting risk assessment may be state independent of state dependent.
In other words, when we evaluate $X_a$ conditionally, it is understood as a state dependent random loss.
We assume that this random loss is continuous in $a$ and $\omega$.
This notation facilitates us in dropping $\omega$ in the notation while still rendering the loss as random and in considering risk quantification conditional on a state $t$.
Nevertheless, one can regard this notation as defining $X_a(\omega)=l(a, \omega)$ for a given continuous loss function $l$.

Since risk can be evaluated both ex ante and interim, actions of the receiver are also feasible at the two decision stages.
Accordingly, the sender's incentive for persuasion may lie in shaping the expectation of average loss based on the actions taken at either stage.

Suppose that the sender's loss depends on the action taken by the receiver based on the average revised conditional risks.
Then, each belief leads to one action taken before state information is revealed, \ie, $a=a(\mu)$.
Sender's problem can be written as
\begin{equation}
    \begin{aligned}
        \inf_{\eta, a(\mu)} \quad & \mathbb{E}_{\eta}\mathbb{E}_{\mu}[v(a(\mu),t)]
        \\
        \text{s.t.} \quad 
        &a(\mu)\in \argmin_{a\in A}
        \mathbb{E}_{\mu}\left[ \text{AV@R}_{\alpha,Z_t^*(\mu)}(X_a|t)\right],
        \\
        & \{Z_t^*\}_{t\in T}\in \mathcal{Z}^*(\mu),
        \\
        &
        \sum_{\text{supp}(\eta)}\mu \eta(\mu)=\mu_0.
        \label{eq:persuasion with one action per belief}
    \end{aligned}
\end{equation}

A variant of problem (\ref{eq:persuasion with one action per belief}) can also be considered.
If sender's loss depends on the receiver's state-dependent actions taken based on each revised conditional risk, we need to consider the collection of actions $\{a_t(\mu)\}_{t\in T}$.
Then, each state leads to a state-dependent action and each belief induces a profile of state-dependent actions.
Without loss of generality, we assume that $a_t\in A$ for all $t\in T$.
An example of this scenario is when the average action $\Bar{a}=\mathbb{E}_\mu[a_t(\mu)]$ enters the sender's loss function $v(\Bar{a}, t)$.
We will also use the notation $v(a, t)$ when the action is understood as the collection of $a_t$ and write $v(\{a_t(\mu)\}_{t\in T,},t)$
.
Note that $a_t(\mu)$ depends on $\mu$, as risk preference revision depends on the mixture distribution over $Q=\mu\circ P$ induced by belief $\mu$.
Sender's problem in this case can be formulated as 
\begin{equation}
    \begin{aligned}
        \inf_{\eta, \{a_t(\mu)\}_{t\in T}} \quad & \mathbb{E}_{\eta}\mathbb{E}_{\mu}[v(\{a_t(\mu)\}_{t\in T},t)]
        \\
        \text{s.t.} \quad 
        &a_t(\mu)\in \argmin_{a_t\in A} 
        \text{AV@R}_{\alpha,Z_t^*(\mu)}(X_a|t), \forall t\in T
        \\
        & \{Z_t^*\}_{t\in T}\in \mathcal{Z}^*(\mu),
        \\
        &
        \sum_{\text{supp}(\eta)}\mu \eta(\mu)=\mu_0.
        \label{eq:persuasion with one action per state}
    \end{aligned}
\end{equation}
In problems (\ref{eq:persuasion with one action per belief}) and (\ref{eq:persuasion with one action per state}), the risk preference revision characterized by $Z_t^*$ can be understood as a side effect of the persuaded posterior beliefs.
However, this side effect enriches the design problem by introducing an additional degree of freedom in how induced beliefs operate, which the designer can exploit to achieve a more effective outcome.
We will illustrate this feature in the application presented in Section \ref{sec:application}.

Several remarks on the sender's problems (\ref{eq:persuasion with one action per belief}) and (\ref{eq:persuasion with one action per state}) are in order. 
First, we assume that problems (\ref{eq:persuasion with one action per belief}) and (\ref{eq:persuasion with one action per state}) differ only in the structure of the sender's loss but not in how the revised risk preferences of the receiver is obtained. 
Second, as we have discussed previously, the optimal dual variables $Z_t^*$ is dependent on the specific random loss vector considered in the dual representation.
Since actions of the receiver is not ignored, we assume that the revision is performed based on the optimal ex ante decision before observation of states is made (hence before revisions take place). 
When multiplicity of optimal ex ante actions occurs, the sender can pick any one that he prefers.
Note that this assumption is for the purpose of making the sender's problems well-defined. 
Analysis performed in later sections rely on the uniqueness of the optimal dual variables, which makes this consideration irrelevant.
Third, persuasion, or the signaling procedure, has to conclude before the reveal of states. 
Thus, we also assume that risk preference revision is performed with respect to ex ante action that is independent of states and is obtained based on the mixture distribution induced by the prior belief.
Note that this is partly due to our motivation that persuasion aims at manipulating the beliefs about the states, whose effect degenerates after one observes specific realizations of states.  
There is no general rule on the selection of the random loss vector for determining preference revision.
Forth, one of the motivations for considering state-dependent actions in (\ref{eq:persuasion with one action per state}) lies in the following relation:
\begin{equation*}
    \begin{aligned}
        \mathbb{E}_\mu \Big[\min_{a_t\in A}\text{AV@R}_{\alpha, Z_t^*(\mu)}(X_{a_t}|t)\Big]
        \leq
        \min_{a\in A}\mathbb{E}_\mu \Big[\text{AV@R}_{\alpha, Z_t^*(\mu)}(X_{a}|t)\Big]
        \leq
        \min_{a\in A} \text{AV@R}_{\alpha}(X_a),
    \end{aligned}
\end{equation*}
where the last inequality follows from Theorem 20 of \citet{pflug2016time}.
This relation renders state-dependent actions preferred by the receiver.
Another motivation for the consideration lies in the fact that the one receiver can be seen as evolving to generate copies of herself in each state, with each copy having a specific objective function.
However, in standard persuasion problems, this feature is absent since the preference stays the same before and after the reveal of states.

\section{Preference persuasion per se}
\label{sec:analysis:per se}

\label{sec:analysis without action}
We first investigate sender's problem (\ref{eq:persuasion without action}) where the actual action taken by the receiver is ignored. 

\subsection{Bounding the effect of a signal rule}
\label{sec:analysis:multiplicity}
Given receiver's belief $\mu$, her preference revision is captured by 
\begin{equation}
    \sup_{Z_t} \mathbb{E}_{\mu}[Z_t\cdot \text{AV@R}_{\alpha, Z_t}(Y|t)],
    \label{eq:decomposition given mu}
\end{equation}
where the supremum is over the set $\mathcal{Z}(\mu):=\{Z_t, t\in T: \mathbb{E}[Z_t]=1, Z_t\geq 0, (1-\alpha)Z_t\leq \1\}$.
%Let $\mathcal{Z}^*(\mu)$ denote the set of optimizers $Z_t^*\subset \mathcal{Z}(\mu)$ for $t\in T$ of (\ref{eq:decomposition given mu}).
The set of optimizers  $\mathcal{Z}^*(\mu)$ is not a singleton in general (see \citet{pflug2016time}).

Define $f(\{Z_t\}_{t\in T}, \mu)=\mathbb{E}_\mu[Z_t\cdot \text{AV@R}_{\alpha, Z_t}(Y|t)]$.
Since $\text{AV@R}_\alpha(\cdot)$ is continuous in the confidence level $\alpha$, $f$ is continuous. 
The feasible set $\mathcal{Z}(\cdot)$ is a compact-valued correspondence such that $\mathcal{Z}(\mu)\neq \emptyset$ for all $\mu\in \Delta(T)$, since $\1\in \mathcal{Z}(\mu)$ for all $\mu$. 
Let $f^*(\mu)$ denote the value of $f$ evaluated at the optimizers of (\ref{eq:decomposition given mu}).
Let $\mathcal{Z}^*(\mu):=\{\{Z_t\}_{t\in T}\in \mathcal{Z}(\mu): f(\{Z_t\}_{t\in T}, \mu)=f^*(\mu)\}$.
The following property facilitates out analysis.

\begin{lemma}
\label{lemma:hemicontinuity of Z(mu)} 
Suppose that every belief $\mu$ is fully-supported.
Then, the correspondence $\mathcal{Z}(\cdot)$ is continuous.
\end{lemma}

Then, since by definition $\mathcal{Z}(\cdot)$ is compact-valued, the facts that $f$ is continuous and $\mathcal{Z}(\cdot)$ is a continuous correspondence indicate via Berge's theorem that $f^*$ is continuous and $\mathcal{Z}^*(\cdot)$ is upper hemicontinuous with nonempty and compact values.

Note that for a posterior belief to be fully supported, we can adopt the standard assumption that the prior belief $\mu_0$ is fully-supported and, in addition, restrict the feasible set of signal rules such that the likelihood function in the Bayes update (\ref{eq:Bayes rule}) is strictly positive everywhere across the state space $T$.
An alternative to the condition that beliefs are fully-supported is assuming that the subset of states with positive probability is fixed. 
Both of these two conditions assure that zero-probability states remain completely constant, thereby making the correspondence $Z(\cdot)$ behave continuously over the restricted subsimplex contained in $\Delta(T)$. 
An interpretation to the requirement that signal rules being strictly positive lies in that persuasion cannot make the DM regard a previously possible state as impossible, or vice versa.

Due to the bi-level structure of the sender's problem (\ref{eq:persuasion without action}), a non-singleton set $\mathcal{Z}^*$ may lead to non-existence of solutions to the sender's problem.
This multiplicity issue is often addressed by adopting sender preferred subgame perfect equilibrium, which allows the sender perform selection in his favor. 

For belief $\mu$, let $v(\mu):=\mathbb{E}_{\mu}[\text{AV@R}_{\alpha, Z_t^*}(Y|t)]$ denote sender's expected loss given $\{Z_t^*\}_{t\in T}$ under this belief. 
Sender preferred subgame perfection is enabled by adopting the selection defined as
\begin{equation}
    \hat{v}(\mu):=\min_{\{Z_t^*\}\in \mathcal{Z}^*(\mu)} \mathbb{E}_{\mu}[\text{AV@R}_{\alpha, Z_t^*}(Y|t)],
    \label{eq:sender preferred selection}
\end{equation}
which is well-defined due to the nonemptyness and compactness of $\mathcal{Z}^*(\mu)$ and continuity of $\text{AV@R}$ in its confidence level.
Let the optimizer of (\ref{eq:sender preferred selection}) be denoted $\{\hat{Z}_t^*\}_{t\in T}$.
This selection criterion indicates that when the receiver is indifferent between optimal risk adjustments in $\mathcal{Z}^*(\mu)$, the sender determines which adjustment will be adopted. 
If the sender is also indifferent between alternatives that are optimal solutions to problem (\ref{eq:sender preferred selection}), he can use arbitrary tie-breaking rules. 
This setting leads to the following result, which is analogous to the one established in \citet{kamenica2011bayesian}.

\begin{lemma}
\label{lemma: v hat lower semi}
$\hat{v}(\cdot)$ is lower semicontinuous.
\end{lemma}

The above lower semicontinuity property facilitates the minimization of the expectation of $\hat{v}$ over the distribution of posterior beliefs, \ie, $\inf_{\eta} \mathbb{E}_{\eta}\hat{v}(\mu)$.
However, as we incorporated the assumption of fully-supported beleifs in Lemma \ref{lemma:hemicontinuity of Z(mu)}, the feasible distribution of posterior beliefs also excludes the boundary.
This leads to the non-existence of minimizer.
Nevertheless, if we extend the minimization problem to be considered in the closure of the distribution of fully-supported posterior beliefs, then it is equivalent that we are minimizing over the entire simplex, which is naturally a compact set.
Hence, a valid tight lower bound can be characterized by the following optimization problem.

\begin{corollary}
\label{coro:lower bound}
The value of a signal rule is lower-bounded by the solution to the following optimization problem
\begin{equation}
\begin{aligned}
    \min_{\eta\in \Delta(\Delta(T))} \mathbb{E}_{\eta}&\hat{v}(\mu)
    \\
    \text{s.t. } \sum_{\text{supp}(\eta)}&\mu\eta(\mu)=\mu_0.
    \label{eq:bound on optimal value}
\end{aligned}
\end{equation}
\end{corollary}

Following the literature on information design, the following approach can be adopted to identify the optimal value of sender's problem.
Let $\text{epi}(\hat{v})$ denote the epigraph of $\hat{v}$.
Let $\text{cov}(\text{epi}(\hat{v}))$ denote the convex hull of the epigraph.
Note that since the sender is a minimizer according to our modeling assumption, here we resort to the epigraph instead of the hypograph.
For all $\mu\in\Delta(T)$, let $V(\cdot)$ denote the largest convex function that is dominated by $\hat{v}$:
\begin{equation}
    V(\mu):=\inf\{b:(\mu, b) \in \text{cov}(\text{epi}(\hat{v}))\}.
    \label{eq:V(mu)}
\end{equation}
The convexity of $V$ follows by its definition.
Since $\hat{v}$ is lower semicontinuous, any element of the graph of $V$ can be expressed as a convex combination of elements of the graph of $\hat{v}$.
This leads to the existence of a Bayes-plausible distribution of posterior beliefs $\eta$ that attains sender's optimal value $\mathbb{E}_{\eta}\hat{v}(\mu)=V(\mu_0)$.
Accordingly, with the knowledge of $\eta$ and $\mu_0$, the optimal signal rule can be constructed as 
\begin{equation}
    \pi(s|t)=\frac{\mu_s(t)\eta(\mu_s)}{\mu_0(t)}.
    \label{eq:optimal signal rule}
\end{equation}
In the sequel, we will occasionally write $\eta_s$ for $\eta(\mu_s)$.

Note that because we require posterior beliefs to be fully-supported, the optimal solution to (\ref{eq:bound on optimal value}) may not be valid.
Consequently, the points needed to write down the convex combination may not be achievable, \ie, they lie on the boundary of the set of posterior beliefs.

The plausibility requirement may also prevent the optimal distribution of posterior beliefs from using boundary distributions, as it restricts how much probability mass can shift.
In practice, an $\epsilon$-approximate solution to (\ref{eq:bound on optimal value}) for a small enough $\epsilon$  is often sufficient to carry out the intended persuasion.
One may also consider adopting a regularized objective in (\ref{eq:bound on optimal value}) by adding a barrier function (such as a function whose value goes to infinity on the boundary) to penalize boundary distributions.
Then, one can force the solution to lie in the interior of the feasible set.

In the appendix of this paper (Section \ref{sec:example}), we provide a example to illustrate how the designer can aim to decrease or increase the average risk in the preference persuasion per se model. 
We will also show how corresponding signal rules can be constructed to approach the bound in Corollary \ref{coro:lower bound} and discuss potential applications depending on the alignment of the objectives of the sender and the receiver.

\subsection{When is preference persuasion beneficial?}
\label{sec:analysis increase average risk}

The bound obtained in the previous section serves as a theoretical limit of the effect of preference persuasion.
An alternative and perhaps more practical question to investigate is when the performance of preference persuasion is sufficient from the perspective of the sender.
To this end, we aim to investigate under what condition can a signal rule $\pi(\cdot|t)$, or the induced distribution of posterior beliefs $\eta$, decrease the average risk under revised preferences evaluated ex ante. 
This question is in the same spirit as asking when can the sender benefit from performing persuasion, which is one of the central questions in the information design literature.
However, our focus on influencing the receiver’s risk preferences gives rise to two key distinctions.
The first difference concerns the objective.
We seek to determine whether revised risk preferences can benefit the sender, whereas standard persuasion models focus primarily on inducing changes in actions.
The second difference is technical in nature.
If we regard the optimal dual variables $\{Z_t^*\}_{t\in T}$ that induce preference revisions as the ``actions" of the receiver, then they lie in a continuous set.
Accordingly, it is not feasible to locate a neighborhood of a belief in which the induced ``actions" stay the same. 
Without this property, the standard approach to establishing benefits for the sender does not apply in our setting.
Thus, we need to resort to a different strategy to derive the intended outcome.

As in the literature, we also impose the condition that “there is information that the sender would share.”
In our setting, this requires the existence of a belief $\mu\in\Delta(T)$ such that 
\begin{equation}
    \mathbb{E}_\mu[\text{AV@R}_{\alpha, Z_t^*(\mu)}(Y|t)]
    <
    \mathbb{E}_{\mu}[\text{AV@R}_{\alpha, Z_t^*(\mu_0)}(Y|t)]-\epsilon,
    \label{eq:information sender would share}
\end{equation}
where we use $Z_t^*(\mu)$ and $Z_t^*(\mu_0)$ to denote optimal solutions to (\ref{eq:decomposition given mu}) given beliefs $\mu$ and $\mu_0$, and $\epsilon>0$ is a positive constant.
Condition (\ref{eq:information sender would share}) means that if there is private information that makes the sender to believe $\mu$, he would prefer to share it with the receiver rather than holding it privately.

If there is no information that the sender would share, \ie, condition (\ref{eq:information sender would share}) does not hold for any choice of positive constant $\epsilon$, then the sender cannot benefit from preference persuasion.
This claim follows from the fact that, for a signal rule $\pi$ that induces $\eta$, it generates sender's value
\begin{equation*}
\begin{aligned}
         \sum_{s\in S} \eta_s
        \left( \mathbb{E}_{\mu_s}[\text{AV@R}_{\alpha, Z_t^*(\mu_s)}(Y|t)] \right)
        &\geq  
        \sum_{s\in S} \eta_s
        \left( \mathbb{E}_{\mu_s}[\text{AV@R}_{\alpha, Z_t^*(\mu_0)}(Y|t)] \right) 
        \\
        &= 
        \mathbb{E}_{\mu_0}[\text{AV@R}_{\alpha, Z_t^*(\mu_0)}(Y|t)],
\end{aligned}
\end{equation*}
where the last equality is due to Bayes-plausibility.

Condition (\ref{eq:information sender would share}) alone is not sufficient to establish benefit of the sender.
By Bayes plausibility, the prior belief must be a convex combination of the induced posterior beliefs. 
Consequently, every posterior belief that benefits the sender in the sense of condition (\ref{eq:information sender would share}) must be paired with (at least) one other belief so that their average coincides with the prior.
See, e.g., \citet{koessler2022interactive}, for in-depth discussions on the splitting of the prior belief.
However, this accompanying belief need not itself be one that the sender would choose to disclose in general.
Therefore, we aim to quantify the deviation in the sender’s objective induced by this accompanying belief, so that any gain from information the sender would share outweighs the potential negative effect it entails.

The following constructions are helpful to us. 
\begin{definition}
\label{def:wasserstein}
Let $W_1: \Delta(\Omega)\times \Delta(\Omega)\rightarrow \mathbb{R}$ denote the order-$1$ Wasserstein distance between two probability measures in $\Delta(\Omega)$ defined for $Q$ and $Q'$ as
\begin{equation}
    W_1(Q, Q'):=\inf_{\Lambda}\iint_{\Omega\times\Omega}\dd(\omega,\omega')\Lambda(d\omega, d\omega'),
    \label{eq: W1 distance}
\end{equation}
where $\Lambda$ lies in the set of all joint probability measures in $\Delta(\Omega\times\Omega)$ having marginals $Q$ and $Q'$. The choice of $\Lambda$ that attains the infimum of (\ref{eq: W1 distance}) is referred to as the optimal transport plan.
\end{definition}
For $\mu$ and $\mu'$ in $\Delta(T)$, we will also denote their order-$1$ Wasserstein distance as $W_1(\mu, \mu')$, which is defined in a similar way as in Definition \ref{def:wasserstein}.
In addition, we use $\dd_{TV}(\cdot, \cdot)$ to denote the total variation distance between two probability measures on the same measurable space.

In the following, we derive an upperbound on $W_1(Q, Q')$ in terms of $W_1(\mu, \mu')$ when $Q$ and $Q'$ are mixture distributions defined by $Q=\mu\circ P$ and $Q'=\mu'\circ P$.
Our method is inspired by \citet{lin2023integrated} where the relation between the Wasserstein distance of joint probability measures and the Wasserstein distance of stochastic kernels are derived given a fixed mutual marginal.
However, in our setting, the stochastic kernel is fixed while the marginals are allowed to be different.

We first present a technical result. 
Let $\Phi:=\mu\otimes P$ denote the joint probability measure on $\Delta(T\times \Omega)$ defined as 
\begin{equation*}
   \Phi(\{t\}, A)=\mu(t)P(A|t), \forall t\in T \text{ and measurable set } A\subseteq \Omega.
\end{equation*}

\begin{lemma}
\label{lemma: W(Q, Q')<W(mu,mu')+ const}
For $\mu, \mu' \in \Delta(T)$, let $Q=\mu\circ P$ and $Q'=\mu'\circ P$ denote the mixture distributions. 
Let $G(\mu, \mu'):=\int_{T\times T}W_1(P(\cdot|t), P(\cdot|t'))\Gamma_{\mu, \mu'}^*(dt, dt')$, where $\Gamma_{\mu, \mu'}^*\in \Delta(T\times T)$ denotes the optimal transport plan of associated with marginals $\mu$ and $\mu'$.
Then, $W_1(Q, Q')\leq W_1(\mu, \mu') + G(\mu, \mu')$.
\end{lemma}

We write $\text{AV@R}_{\alpha}^{Q}(Y)$ to indicate that the risk of $Y$ is evaluated by $\text{AV@R}$ at confidence level $\alpha$ under the reference probability measure $Q\in \Delta(\Omega)$.
The following technical lemma provides the sensitivity of AV@R with respect to the underlying reference probability measure (see, e.g., \citet{pflug2014multistage}).
\begin{lemma}
\label{lemma:e2 bound}
For $Q_1, Q_2 \in \Delta(\Omega)$, it holds that
\begin{equation*}
    \Big| \text{AV@R}_{\alpha}^{Q_1}(Y)-\text{AV@R}_{\alpha}^{Q_1}(Y)\Big|
    \leq 
    \frac{1}{1-\alpha}\cdot L(Y)\cdot W_1(Q_1, Q_2),
\end{equation*}
where $L(Y)$ is the Lipschitz constant such that $|Y(\omega)-Y(\omega')|\leq L(Y)\cdot \dd(\omega, \omega')$.
\end{lemma}

%\textcolor{blue}{****maybe need to say this earlier*****}Here, the semi-distance $\dd(\omega, \omega'):=||\xi(\omega)-\xi(\omega')||$ for some random variable $\xi:\Omega\rightarrow \mathbb{R}$ and some norm $||\cdot||$ in $\mathbb{R}$.

The following result concerns the deviation of the sender's loss.
\begin{lemma}
\label{lemma:v hat bound}
Let $G$ be defined in Lemma \ref{lemma: W(Q, Q')<W(mu,mu')+ const} and $L(Y)$ be defined in Lemma \ref{lemma:e2 bound}.
Furthermore, let $M_0=\max_t \Big|\text{AV@R}_{\alpha, \hat{Z}_t^*(\mu_0)}(Y|t)\Big|$ and $M'=\max_t \Big|\text{AV@R}_{\alpha, \hat{Z}_t^*(\mu')}(Y|t)\Big|$.
Let $m=\max \{1, \frac{\alpha}{1-\alpha}\}$.
Then, it holds for $\mu\in\Delta(T)$ that
\begin{equation*}
\begin{aligned}
    &\big|\mathbb{E}_{\mu}[\text{AV@R}_{\alpha, \hat{Z}_t^*(\mu)}(Y|t)]
    -
    \mathbb{E}_{\mu}[\text{AV@R}_{\alpha, \hat{Z}_t^*(\mu_0)}(Y|t)]\big|
    \\
    & \qquad \leq \frac{L(Y)}{1-\alpha}\cdot \Big(W_1(\mu, \mu_0)+G(\mu, \mu_0) \Big)+ M_0(2\dd_{TV}(\mu, \mu_0)+m)+ M'm.
\end{aligned}
\end{equation*}
\end{lemma}

Integrating the above analysis, we arrive at the following result of the benefit of the sender.
\begin{theorem}
\label{thm:sender benefit no action}
Suppose that the conditions in Lemma \ref{lemma:v hat bound} are satisfied. 
Then, risk preference persuasion can decrease average risk if
\begin{equation*}
    \epsilon>\frac{L(Y)}{1-\alpha}\cdot \Big(W_1(\mu_0,\mu')+G(\mu_0,\mu')\Big)+M_0(2\dd_{TV}(\mu_0, \mu')+m)+M'm.
\end{equation*}
\end{theorem}

\subsection{Size of signal space}
\label{sec:signal space}

The formulation of the sender's problem in (\ref{eq:persuasion without action}) relies on the set of posterior beliefs to search for and reconstruct the optimal signal rule. 
Although this set is slightly simpler in structure than the set of joint probability distributions over the product set of signals and states, it is still very large.

\citet{kamenica2011bayesian} showed that the maximum number of signals that the sender requires to achieve optimal persuasion is upperbounded by the size of the state space.
With  the availability of the revelation principal, these signals can be directly realized as action recommendations. 
While revelation principal fails when the receiver deviates from risk-neutrality, a similar argument applies to belief recommendation in place of action recommendation (see \citet{anunrojwong2024persuading}).
Therefore, based on our formulation using the distribution of posterior beliefs, representing $(\mu_0, V(\mu_0))$ using a convex combination of elements from $\text{epi}(\hat{v})$ requires at most $|T|$ elements. 
This means that the optimal signal rule needs to contain at most $|T|$ distinct signals.

Note that direct ``action" recommendation does apply to problem (\ref{eq:persuasion without action}), as the receiver can be considered as an expected utility maximizer if we view the dual variable $Z_t$ as her ``action".
This observation can be deduced from the receiver's problem (\ref{eq:decomposition given mu}), as it is a linear function of the belief $\mu$ given fixed $Z_t$. 
Here, we apply the belief recommendation argument as it is feasible in broader contexts such as in problems (\ref{eq:persuasion with one action per belief}) and (\ref{eq:persuasion with one action per state}) where the objective functions of the receiver lack linearity in belief.

%Furthermore, \citet{anunrojwong2024persuading} also derived an alternative upperbound on the cardinality of the signal space based on the splitting of the probability associated with the receiver playing a specific action at a given state. 
%We refer to the online companion of \citet{anunrojwong2024persuading} for more detailed discussions on the relation between the two approaches.
%We also refer the reader to \citet{koessler2022interactive} for more discussions on the size of the signal space and to \citet{aumann1995repeated} for the foundation on the splitting of  beliefs.

\section{Preference persuasion with actions}
\label{sec:analysis with action}
In this section, we investigate the existence of optimal solutions to the sender's problems when the action of the receiver is taken into account described by (\ref{eq:persuasion with one action per belief}) and (\ref{eq:persuasion with one action per state}) together.

Similar as in Section \ref{sec:analysis:multiplicity}, we allow the sender to choose a preferred action in case where the optimal solution set to the receiver's problem is not a singleton.
However, since the receiver, upon generating a posterior belief, first updates her risk preferences contingent on state information then makes a decision, the sender's objective value, as a function of the belief, needs to take into account the influence of revised preferences on the actions taken. 

To avoid further complicating the persuasion problems (\ref{eq:persuasion with one action per belief}) and (\ref{eq:persuasion with one action per state}), which by themselves involve two layers of decision-making, the following result employs an assumption on the uniqueness of the optimal dual variable associated with the risk preference revision.

\begin{theorem}
 \label{thm:existence of optimal value with action}  
Suppose that the optimal dual variable $Z^*$ used for generating the revised preferences is the unique solution to the dual representation of the sender's original risk preference at the ex ante random loss.
Then, there exists optimal solution to the sender's problems (\ref{eq:persuasion with one action per belief}) and (\ref{eq:persuasion with one action per state}).
\end{theorem}

The existence of optimal solutions is not guaranteed in general once the assumption of uniqueness of the optimal dual variable is relaxed. The underlying reason is that the set‑valued mapping $\mathcal{Z}^*(\mu)$, which represents the set of optimal solutions to an optimization problem, need not be lower hemicontinuous. Consequently, a continuous selection of this correspondence may fail to exist.
Without the continuity property, we know from the proof of Theorem \ref{thm:existence of optimal value with action} that the set of optimal solutions to the receiver's problem may lack the required upper hemicontinuity. 
In addition, the condition that posterior beliefs must be fully-supported is irrelevant due to this uniqueness assumption.

Note that when receiver's actions are taken into account, one can also investigate the conditions under which the sender would benefit from persuasion.
As we have assumed that the action set $A$ is finite, the analysis would be same as in standard Bayesian persuasion setting, provided that the assumption of the uniqueness of optimal dual variables holds.
When the action set is infinite, we may extend the approach in Section \ref{sec:analysis increase average risk} by employing additional assumptions on the monotonicity and growth of the loss functions of the sender and receiver and derive conditions on when the sender could benefit from persuasion.
However, the analysis would be lengthy and unlikely to yield any additional insights.

Finally, we remark here that, based on the belief recommendation argument discussed in Section \ref{sec:signal space}, the maximum number of signals that induced optimal persuasion in problems (\ref{eq:persuasion with one action per belief}) and (\ref{eq:persuasion with one action per state}) are upperbounded by the cardinality of the state space $|T|$.
Since problems (\ref{eq:persuasion with one action per belief}) and (\ref{eq:persuasion with one action per state}) have objectives that are nonlinear in beliefs, direct action recommendations cannot be used to construct optimal persuasion in general.
We refer the reader to the running example in \citet{anunrojwong2024persuading} for how nonlinearity in the receiver's objective function prevents using coalescence of actions.

\section{Application}
\label{sec:application}

In this section, we apply the Bayesian risk preference persuasion framework to a reinsurance design problem and investigate how persuaded risk preferences facilitate the reinsurer's design.

Reinsurance is a traditional financial mechanism that enables risk-sharing at the level of insurers. 
A reinsurance contract contains a premium payment to the reinsurer payed by the participating insurers in exchange for the reinsurer's coverage of a portion of the random financial losses faced by the insurers.
As a pre-incident risk management strategy, reinsurance shields insurers from insolvency in the event of widespread or catastrophic disasters, thereby enhancing the overall stability of the financial market.
Risk measures have been widely used to capture the insurer's preferences (see, e.g., \citet{cai2020optimal,chi2011optimal}).
However, the instability of risk preferences has been overlooked in standard models.
A related recent work by \citet{su2025continuous} investigates a reinsurance problem in which risk preferences are known up to incomplete information. 
They have focused on the competition among multiple reinsurers.

We build on the literature but  emphasize the role of information design in shaping insurers' risk preferences for enhancing reinsurance design. 
We adopt the setting of problem (\ref{eq:persuasion with one action per state}) where the sender's objective function depends on the profile of state-dependent actions from the receiver induced by her revised preferences.
To this end, we first introduce elements in a standard reinsurance model.

\subsection{Standard reinsurance design framework}
\label{sec:app:standard model}
Consider the scenario involving one reinsurer and one insurer. 
Let $X$ denote the random loss.
The objective is to design a reinsurance indemnity $I(X)$ that determines the amount of risk ceded to the reinsurer.
Assume that admissible indemnity functions lie in the set 
\begin{equation*}
    \mathcal{I}:=\{I: [0, \esssup(X))\rightarrow \mathbb{R}_+ | I(0)=0 \text{ and } 0\leq I(x)-I(y)\leq x-y, \forall 0\leq y\leq x \}.
\end{equation*}
This choice is commonly adopted (see, e.g., \citet{chi2011optimal, su2025continuous}) and admissible indemnity function $I(\cdot)\in \mathcal{I}$ prevents moral hazard issues. 

Given the loss ceded to the reinsurer $I(\cdot)$, the financial loss faced by the reinsurer and the insurer are $I(X)-h(I(X))$ and $X+h(I(X))-I(X)$, respectively, where $h(\cdot)$ determines the premium payment charged.
We adopt the expected-value premium principal represented by 
\begin{equation*}
    h(I(X))=(1+\kappa)\mathbb{E}[I(X)]
\end{equation*}
with a safety loading coefficient $\kappa>0$.
There are other commonly-used payment rules, such as the variance  premium principle and the standard deviation premium principle.
However, we use the expected-value premium principal to keep the model simple and to focus on preference persuasion.

Suppose that the insurer's risk preference is described by a risk measure $\text{AV@R}_\alpha(\cdot)$.
The reinsurance design problem can be summarized as 
\begin{equation}
    \min_{I\in\mathcal{I}} \text{AV@R}_\alpha(X+h(I(X))-I(X)).
    \label{eq:general reinsurance design problem}
\end{equation}
From \citet{chi2011optimal}, we know that the optimal solution to (\ref{eq:general reinsurance design problem}) is 
\begin{equation}
    I^*(x)=\begin{cases}
        (x- \text{V@R}_{\frac{\kappa}{1+\kappa}}(X))_+, & \text{ if } \alpha>\frac{\kappa}{1+\kappa},
        \\
        0, & \text{ if } \alpha\leq \frac{\kappa}{1+\kappa},
    \end{cases}
    \label{eq:optimal I of general insurance}
\end{equation}
with corresponding optimal risk represented by 
\begin{equation}
    \text{AV@R}_\alpha(X+h(I^*(X))-I^*(X))=\begin{cases}
        b^*, & \text{ if } \alpha>\frac{\kappa}{1+\kappa},
        \\
        \text{AV@R}_\alpha(X), & \text{ if }\alpha\leq \frac{\kappa}{1+\kappa},
    \end{cases}
    \label{eq:optimal risk given I*}
\end{equation}
where $b^*=d^*+(1+\kappa)\mathbb{E}[(X-d^*)_+]$ with $d^*=\text{V@R}_{\frac{\kappa}{1+\kappa}}(X)$.

\subsection{Reinsurance with preference persuasion}
\label{app:persuasion in reinsurance}
Suppose that the reinsurer can resort to risk preference persuasion, in addition to the contract, to aid his design of the reinsurance plan.
Building on the above standard reinsurance framework, we introduce a state space containing two distinct states $T=\{t_1 ,t_2\}$.
Conditional on a state $t\in T$, we assume that the random loss $X$ follows an exponential distribution whose probability density function is denoted $P(\cdot|t)$ with known rate parameter $t$.
Thus, given $t\in T$, quantification of the randomness $X$ is performed with respect to $P(\cdot|t)$, which is a conditional distribution as illustrated by the notation.
For example, under state $t_1$, the expectation of the random loss is $\mathbb{E}[X|t_1]=\frac{1}{t_1}$ due to properties of the exponential distribution. 
%This is the model setting adopted in Section \ref{sec:framework}.
Without loss of generality, we assume that $0<t_1<t_2$.
Since $t$ represents the rate parameter, this assumption indicates that state $t_1$ is more risky than $t_2$ as $P(\cdot|t_1)$ stochastically dominates $P(\cdot|t_2)$.

At the ex ante stage in which an observation of the state has not been made yet, risk is evaluated on the basis of a belief about the likelihood of the states. 
Suppose that a belief $\mu=(q, 1-q)$ for $0\leq q\leq 1$ is given.
Then, risks associated with the randomloss $X$ is evaluated with respect to the mixture distribution $\mu\circ P$, which admits a density function
\begin{equation*}
    f_X(x)=qt_1e^{-t_1x}+(1-q)t_2e^{-t_2x}, \text{ for } x\in [0,+\infty).
\end{equation*}
The unconditional expectation under this distribution is denoted $\mathbb{E}[X]=\frac{q}{t_1}+\frac{1-q}{t_2}$.

The reinsurer is risk-neutral and uses the expectation risk measure $\mathbb{E}(\cdot)$ to evaluate risks. 
The insurer, on the other hand, is assumed to be risk-averse and her initial risk preference is assume to be $\text{AV@R}_\alpha$.
Whether the respective reference probability measure is conditional on an observed state depends on when the risk quantification is performed. 
This timing will be clear from our notations.

\paragraph{Preference revision given belief.}
Preference revision depends on insurer's initial preference, the belief about the states which induces the mixture distribution, and the ex ante action which determines the loss. 
Recall that, in Section \ref{sec:analysis with action}, we assumed that the ex ante action that determines the loss vector for preference revision is the optimal state-independent action.
We maintain his assumption in this application.

Let $\nu=\text{V@R}_\alpha(X)$.
Using the probability density $f_X(x)$ and properties of the exponential distribution, $\nu$ can be shown to satisfy
\begin{equation}
    qe^{-t_1\nu}+(1-q)e^{-t_2\nu}=1-\alpha.
    \label{eq:nu satisfy this eq}
\end{equation}
Equation (\ref{eq:nu satisfy this eq}) has no analytical solution in $\nu$ in general.
The next result shows the monotonicity of $\nu$ as a function of the probability $q$.
The proof is presented in the appendix.
\begin{lemma}
\label{lemma:nu is increasing in q}
The value of $\nu$ that satisfies (\ref{eq:nu satisfy this eq}) is strictly increasing in $q$.
\end{lemma}

The following result locates the revised risk preferences of the insurer contingent on observed states.
Despite the fact that the optimal indemnity function (\ref{eq:optimal I of general insurance}) is piecewise depending on the relation between the confidence level and the safety loading coefficient, we show that the revised risk preferences on respective domains coincide.

\begin{lemma}
\label{lemma:revised preferences under all alpha}  
Suppose that the insurer's risk preference is revised based on the ex ante loss vector induced by the optimal ex ante action.
Then, the revised risk preferences are identified with the following confidence levels $\alpha_t$ for $t\in T$:
\begin{equation*}
    \alpha_t=\begin{cases}
        1-e^{t_1\nu}, & \text{ if } t=t_1,
        \\
        1-e^{t_2\nu}, & \text{ if } t=t_2.
    \end{cases}
\end{equation*}
\end{lemma}

\paragraph{Reinsurer's value.}
The revised preferences characterized by $\alpha_t$ for $t\in T$ leads to an interim version of the optimal reinsurance.
Compared to the original problem (\ref{eq:general reinsurance design problem}), the difference lies in that the problem is now concerning the conditional distributions $P(\cdot|t)$ instead of the mixture distribution $\mu\circ P$ and optimal reinsurance is determined contingent on observed state.
Note that if one adopts the most generic model, then there may also include a state-dependent safety loading coefficient $\kappa_t$.
However, to simplify presentation and focus on the design of information, we assume that $\kappa=\kappa_t$ for all $t\in T$.
We will use $I_t(\cdot)$ to denote the state-dependent indemnity function under state $t$, in which the reference probability measure for the V@R and AV@R values is $P(\cdot|t)$.
Then, the optimal state-dependent indemnity function satisfies 
\begin{equation*}
     I_t^*(x)=\begin{cases}
        (x- \text{V@R}_{\frac{\kappa}{1+\kappa}}(X|t))_+, & \text{ if } \alpha_t>\frac{\kappa}{1+\kappa},
        \\
        0, & \text{ if } \alpha_t\leq \frac{\kappa}{1+\kappa}.
    \end{cases}
\end{equation*}
For a given belief $\mu$, the reinsurer's expected loss under optimal indemnity $I_t^*$ is
\begin{equation*}
    \begin{aligned}
        v(\mu)&=\mathbb{E}_\mu
        \Big[ \mathbb{E}_{P(\cdot|t)}[I_t^*(X)-h(I_t^*(X)))] \Big]
        \\
        &=
         \mathbb{E}_\mu
        \Big[ \mathbb{E}_{P(\cdot|t)}\big[
        I_t^*(X)-(1+\kappa)\mathbb{E}_{P(\cdot|t)}[I_t^*(X)]
        \big]\Big]
        \\
        &=-\kappa\mathbb{E}_{\mu}\mathbb{E}_{P(\cdot|t)}[I_t^*(X)].
    \end{aligned}
\end{equation*}

Depending on the value of the revised confidence level $\alpha_t$, we have
\begin{equation}
    \mathbb{E}_{P(\cdot|t)}[I_t^*(X)]=\begin{cases}
        \frac{e^{-td_t^*}}{t}, & \text{ if } \alpha_t>\frac{\kappa}{1+\kappa}, \\
        0, & \text{ if } \alpha_t\leq \frac{\kappa}{1+\kappa},
    \end{cases}
    \label{eq:value is 0 or + based on alpha}
\end{equation}
where $d_t^*=\text{V@R}_{\frac{\kappa}{1+\kappa}}(X|t)=\frac{1}{t}ln(1+\kappa)$.
Thus, $ \mathbb{E}_{P(\cdot|t)}[I_t^*(X)]=\frac{1}{t(1+\kappa)}$ if $\alpha_t>\frac{\kappa}{1+\kappa}$.

As $\frac{1}{t(1+\kappa)}>0$, it is straightforward to observe that the reinsurer prefers to persuade the insurer to increase her revised confidence levels $\alpha_t$ if he aims to minimize his expected loss $v(\mu)$.
Nevertheless, persuasion is constrained by Bayes-plausibility.
It is not immediately obvious whether the reinsurer can benefit from performing persuasion.

The revised confidence level $\alpha_t=1-e^{-t\nu}$ is strictly increasing in $t$ and $\nu$.
As we have assumed that $t_1<t_2$ and the optimal dual variables satisfy $\mathbb{E}[Z_t^*]=1$, we have $\alpha_{t_1}<\alpha<\alpha_{t_2}$.
Consequently, the reinsurer is likely to benefit if the induced beliefs cause more revised confidence levels under state $t_2$ to exceed $\frac{\kappa}{1+\kappa}$ than they cause those under state $t_1$ to fall below $\frac{\kappa}{1+\kappa}$.

In the following, we will investigate whether it is beneficial to perform persuasion and how to design optimal information in detail.

\paragraph{Information design.}

Define the following critical probability values that are useful in the analysis:
\begin{equation*}
    \Bar{q}:=\frac{1-\alpha-(1+\kappa)^{-t_2/t_1}}{\frac{1}{1+\kappa}-(1+\kappa)^{-t_2/t_1}}
    \text{, }
    \underline{q}:=\frac{1-\alpha-\frac{1}{1+\kappa}}{(1+\kappa)^{-t_1/t_2}-\frac{1}{1+\kappa}}
    \text{, and }
    \Tilde{q}:=\frac{1}{1+1/t_1 +(1/\Bar{q}-1)/t_2}.
\end{equation*}
The following lemma establishes a property required to derive the results in this section. Its proof is postponed to the appendix.
\begin{lemma}
\label{lemma:q underline < q bar}
It holds that $\underline{q}<\Bar{q}$.
\end{lemma}

Let $\alpha_t$ for $t\in T$ denote the revised confidence levels induced by the prior belief $\mu_0$.
Due to the piecewise structure of the optimal indemnity function, we investigate expected losses of the reinsurer and the signals that induce the values under the following three cases.
Case (1): $\frac{\kappa}{1+\kappa}<\alpha_{t_1}$; Case (2): $\alpha_{t_1}<\frac{\kappa}{1+\kappa}<\alpha_{t_2}$; and Case (3): $\alpha_{t_2}<\frac{\kappa}{1+\kappa}$.
Conditions on the parameter $q_0$ of the prior belief that lead to these cases are:
in Case (1), $q_0>\Bar{q}$; in Case (2), $\underline{q}<q_0<\Bar{q}$; in Case (3), $q_0<\underline{q}$.
Note that these conditions follows from algebraic manipulation knowing that $\alpha_t$ is specified in Lemma \ref{lemma:revised preferences under all alpha} and $\nu$ satisfies (\ref{eq:nu satisfy this eq}) with $q$ replaced by $q_0$.
To simplify the presentation and reduce the scenarios to analyze, we further assume that $q_0<\Tilde{q}$ in Case (2) and that $t_2>t_1(1/\underline{q}-1/\Bar{q})$ in Case (3).

Similar as in Section \ref{sec:example}, a two-element state spaces requires at most two distinct beliefs recommendations in the optimal signal rule. 
Let the support of the distribution of posterior beliefs induced by the signal rule be $\text{supp}(\eta)=\{\mu', \mu''\}$.
By Bayes-plausibility, the prior will be split as $\mu_0=\gamma\mu'+(1-\gamma)\mu''$ where $\mu_0=(q_0, 1-q_0)$, $\mu'=(q', 1-q')$, and $\mu''=(q'', 1-q'')$.
Hence, we have $q_0=\gamma q'+(1-\gamma)q''$.
We also assume that $q'<q''$ so that belief $\mu''$ assigns higher probability to state $t_1$, the riskier state, than belief $\mu'$. 
We will also refer to $\mu''$ as the higher belief.

If the reinsurer does not perform persuasion, his expected loss is $v(\mu_0)$.
Note that there is still risk preference revision in this scenario and the revision is associated with the prior belief $\mu_0$.
A distribution of beliefs $\eta$ (strictly) benefits the reinsurer  if $\mathbb{E}_\eta v(\mu)<v(\mu_0)$.
The next result summarizes the conditions in the three cases under which the reinsurer benefits from performing risk preference persuasion.

\begin{theorem}
\label{thm:conditions for benefiting reinsurer}  
The following statements hold.
In Case (1), the reinsurer can never strictly benefit and he preserve the original expected loss if $q'\geq \Bar{q}$.
In Case (2), the reinsurer strictly benefits if $q'> \underline{q}$ and $q''> \Bar{q}$.
In Case (3), the reinsurer strictly benefits if $q''>\underline{q}$.
\end{theorem}

The next result summarizes the optimal values of $\mathbb{E}_\eta v(\mu)$ and the signal rules that induce these values under the cases in which the reinsurer benefit from persuasion shown in Theorem \ref{thm:conditions for benefiting reinsurer}.

\begin{theorem}
\label{thm:optimal value in beneficial cases}
Let signal space be $S=\{s', s''\}$ where $s'$ recommends $\mu'$ and $s''$ recommends $\mu''$.
The optimal expected loss of the reinsurer in
Case (2) is $\frac{1}{t_1}\Big(q_0-\underline{q}\frac{\Bar{q}-q_0}{\Bar{q}-\underline{q}} \Big)$, which is enabled by signal rule $\pi(s'|t_1)=\gamma \underline{q}$, $\pi(s''|t_1)=(1-\gamma)\Bar{q}$, $\pi(s'|t_2)=\gamma(1-\underline{q})$, and $\pi(s''|t_2)=(1-\gamma)(1-\Bar{q})$; in Case (3), the loss is $\frac{q_0}{t_1}+\frac{q_0(1-\Bar{q})}{\Bar{q}t_2}$, which is enabled by signal rule $\pi(s'|t_1)=0$, $\pi(s''|t_1)=(1-\gamma)\Bar{q}$, $\pi(s'|t_2)=\gamma$, and $\pi(s''|t_2)=(1-\gamma)(1-\Bar{q})$.
\end{theorem}

\section{Conclusion}
\label{sec:conclusion}
We have proposed a Bayesian risk preference persuasion framework to leverage the instability of human risk preferences for risk management. 
This setting has applications, e.g., varying from system security enhancement to financial risk mitigation. 
However, persuasive information does not always benefit the party aiming to design the signal rules and the optimal design of persuasive information may even be nonexistent. 
To gain more insights, we have investigated two distinct problem formulations.
One aims at elaborating the effect of information on average risk preference.
The other focuses on the end-to-end effect of preference persuasion.
Extending existing analysis approaches, we have identified the value of sender's optimal persuasion with the help of convex analysis and characterized the conditions for benefit based on the decomposition of the transportation distance of mixture distributions. 
We have also applied our theoretical framework to a reinsurance design problem and derived in detail when and how preference persuasion assists the reinsurer.

Some of the future research directions include extending the current framework to the setting where multiple risk-averse receivers interact non-cooperatively and the sender aims to perform preference persuasion to design the equilibrium, developing computational tools for solving the preference persuasion problems, and considering dynamic persuasion policies.

%\appendix
%\begin{APPENDIX}{}%{<Title of the Appendix>}

\begin{appendices}

\section{An example of preference persuasion per se.}
\label{sec:example}
In this section, we present an example to illustrate the spirit of Bayesian risk preference persuasion under the setting where the action of the receiver is ignored discussed in Section \ref{sec:analysis:multiplicity}. 

Consider the state space $T=\{t_1, t_2\}$ containing two distinct states.
Assume that the prior belief is uniform, \ie, $\mu_0=(0.5, 0.5)$.
Given both states,  the conditional probabilities $P(\cdot|t)$ follow a uniform distribution.
The random loss is, given $t_1$, $Y_{|t_1}=[52, 28]$; given $t_2$, $Y_{|t_2}=[60, 0]$.
This setting has the interpretation that state $t_1$ is a more concentrated state while $t_2$ is a more dispersed state.
Suppose that signal $s\in S$ leads to posterior belief $\mu_s=(q,1-q)$ (state  $t_1$ is believed to happen with probability $q$). 
Since posterior beliefs need to be fully supported, the parameter satisfies $0<q<1$.
Then, the mixture distribution $Q=\mu_s\circ P=[\frac{1}{2}q, \frac{1}{2}q, \frac{1}{2}(1-q), \frac{1}{2}(1-q)]$ indicates the probabilities for the loss vector $Y=[52, 28, 60, 0]$.
For instance, this means that distribution $Q$ assigns probability $\frac{1}{2}q$ to loss $28$.

The receiver's initial risk preference is set to $\text{AV@R}_{\frac{1}{3}}$, \ie, the confidence level $\alpha=\frac{1}{3}$.
This means that the worst $\frac{2}{3}$ tail of the random loss will contribute to the risk.
Given the specific $Q$ and $Y$ above, we observe that probability that $Y=52$ or $Y=60$ is $\frac{1}{2}$.
This indicates that the random loss values $52$ and $60$ always contribute to risk quantification, while depending on the value of $q$, $\text{V@R}_{\frac{1}{3}}$ can be either $0$ or $28$.
Thus, in the sequel, we consider these two cases separately with the critical value of $q$ being $\frac{1}{3}$.

\paragraph{Case $1$: $q< \frac{1}{3}$.}
In this case, the optimal dual variable $Z$ satisfying $\text{AV@R}_\frac{1}{3}(Y)=\mathbb{E}_{Q}(YZ)$ is $Z=(\frac{3}{2}, \frac{3}{2},\frac{3}{2},\frac{1-3q}{2(1-q)})$.
This $Z$ satisfies the feasibility requirement in (\ref{eq:dual representation risk measure}), \ie, $Z\geq 0$, $\mathbb{E}_Q(Z)=1$, and $Z\leq \frac{1}{1-\alpha}\1=\frac{3}{2}\1$.
To derive the revised risk preferences, we obtain $Z_{t_1}=\mathbb{E}(Y|t_1)=\frac{3}{2}$ and $Z_{t_2}=\mathbb{E}(Y|t_2)=\frac{2-3q}{2-2q}$.
Then, $\alpha_{t}=1-(1-\alpha)Z_{t}$, we obtain the updated confidence levels $\alpha_{t_1}=0$ and $\alpha_{t_2}=\frac{1}{3(1-q)}$.
The revised risk given $t_1$ follow as $\text{AV@R}_{\alpha_{t_1}}(Y|t_1)=\text{AV@R}_{0}([52, 28]|t_1)=52\times \frac{1}{2}+ 28\times \frac{1}{2}=40$.
Given $t_2$, the revised risk is $\text{AV@R}_{\alpha_{t_2}}(Y|t_2)=\text{AV@R}_{\frac{1}{3(1-q)}}([60, 0]|t_2)$.
Under $q<\frac{1}{3}$, $1-\alpha_{t_2}>\frac{1}{2}$.
Since $P(\cdot|t_2)$ is uniform distribution, we obtain $\text{AV@R}_{\alpha_{t_2}}(Y|t_2)=60\times \frac{1}{2}\times \frac{1}{1-\alpha_{t_2}}+0\times \frac{1}{2}\times (2-\frac{1}{1-\alpha_{t_2}})=\frac{90(1-q)}{2-3q}$.
Therefore, under belief $\mu_s=(q, 1-q)$, the average risk can be computed as 
\begin{equation*}
    \mathbb{E}_{\mu_s}\text{AV@R}_{\alpha_{t}}(Y|t)=q\cdot \text{AV@R}_{\alpha_{t_1}}(Y|t_1)+(1-q)\cdot \text{AV@R}_{\alpha_{t_2}}(Y|t_2) =\frac{-30q^2-100q+90}{2-3q}.
\end{equation*}
Taking the derivative of the above average risk with respect to parameter $q$, we obtain
\begin{equation*}
    \frac{d}{d q}\mathbb{E}_{\mu_s}\text{AV@R}_{\alpha_{t}}(Y|t)=\frac{90[(q-\frac{2}{3})^2+\frac{1}{3}]}{(2-3q)^2}>0,
\end{equation*}
indicating that the average risk is strictly increasing in the belief that $t_1$ is occurring with probability $q$.
Since in this case $0<q<\frac{1}{3}$, the region of average risks can be obtained as $(45, \frac{160}{3})$.

\paragraph{Case $2$: $q\geq \frac{1}{3}$.}
Following the procedure similar as in Case $1$, we obtain $Z=(\frac{3}{2}, \frac{1}{2q}, \frac{3}{2},0)$ with $Z_{t_1}=\frac{3}{4}+\frac{1}{4q}$ and $Z_{t_2}=\frac{3}{4}$.
Thus, updated confidence levels are $\alpha_{t_1}=\frac{1}{2}-\frac{1}{6q}$ and $\alpha_{t_2}=\frac{1}{2}$; revised risks are $\text{AV@R}_{\alpha_{t_1}}=28+\frac{72q}{1+3q}$ and $\text{AV@R}_{\alpha_{t_2}}=60$.
Then, given belief $\mu_s=(q, 1-q)$, the average risk is $\frac{-24q^2+148q+60}{1+3q}$.
Its derivative with respect to $q$ is $\frac{-72[(q+1/3)^2+1/3]}{(1+3q)^2}<0$.
Then, the average risk is strictly decreasing in $q$ for $\frac{1}{3}\leq q<1$.
This yields the region of average risk $[\frac{160}{3}, 46)$.

With the average risks as a function of $q$ in the above two cases, graphical method based on the convex closure of the graph of this function can be leveraged to obtain the optimal value of a sender if his objective is specified. 
In the following, we assume that the sender has two specific objectives: to decrease or to increase the expected average risk.
These two settings indicate that the designer can possess opposite preferences about the expected average risk.
Suppose that we impose the condition that the receiver adopts a threshold strategy, that is, taking a particular action whenever the perceived risk exceeds a given threshold.
Then, minimizing the expected average risk can be interpreted as the sender’s attempt to discourage the receiver from taking that action while maximizing aims at encouraging the action.
We will investigate the signal rules that achieve these two opposite goals.

Suppose that the sender prefers a lower expectation of average risk and he chooses a distribution of posterior beliefs $\eta$ to fulfill this design goal.
This distribution has to satisfy Bayes plausibility for it to be valid.
Observe that the average risk, as a function of $q$, is strictly increasing on $0<q<\frac{1}{3}$ and strictly decreasing on $\frac{1}{3}<q<1$ and it is continuous.
Then, choose $\eta$ with $\text{supp}(\eta)=\{\underline{\mu}, \Bar{\mu}\}$, where $\underline{\mu}=(0,1)$ and $\Bar{\mu}=(1, 0)$ are boundary beliefs.
With $\mu_0=\gamma\underline{\mu}+(1-\gamma)\Bar{\mu}$, we obtain $\gamma=\frac{1}{2}$.
Consequently, $\eta$ splits the prior into $\underline{\mu}$ and $\Bar{\mu}$ with equal probability. 
This leads to the lower bound on the value of the expectation, with respect to $\eta$, of the average risks, whose value is $\gamma\cdot 45 +(1-\gamma)\cdot 46=45.5$.
The signal rule $\pi$ that induces this construction can then be derived from (\ref{eq:optimal signal rule}).
It suffices to use two distinct signals, \ie, $S=\{\underline{s}, \Bar{s}\}$ with $\underline{s}\neq \Bar{s}$ such that $\underline{s}$ recommends belief $\underline{\mu}$ to the receiver and $\Bar{s}$ recommends $\Bar{\mu}$ to the receiver.
The corresponding signal rule is $\pi(\underline{s}|t_1)=\pi(\Bar{s}|t_2)=\frac{1}{2}$ and $\pi(\underline{s}|t_2)=\pi(\Bar{s}|t_1)=0$.
This signaling rule is not feasible since it leads to the boundary beliefs.
Nevertheless, the sender can resort to posterior beliefs that are close to $\underline{\mu}$ and $\Bar{\mu}$ and construct signal rules to realize them.
For example, we can aim at fully-supported posterior beliefs $\underline{\mu}'=(\epsilon, 1-\epsilon)$ and $\Bar{\mu}'=(1-\epsilon, \epsilon)$, where $\epsilon>0$ is chosen sufficiently small.
Due to symmetry in this setting, Bayes plausibility requires to set $\gamma=\frac{1}{2}$ to recover the uniform prior.
If the corresponding signals that recommends these two posterior beliefs are denoted $\underline{s}'$ and $\Bar{s}'$, then the signal rule can be constructed as $\pi(\underline{s}'|t_1)=\pi(\Bar{s}'|t_2)=\epsilon$ and $\pi(\underline{s}'|t_2)=\pi(\Bar{s}'|t_1)=1-\epsilon$.

We note the following observations that serve as a potential intuitive explanation of our assumption on requiring fully-supported posterior beliefs.
To see this, observe that under the signal rule to recommend $\underline{s}$ and $\Bar{s}$, only one state is believed to be possible under each recommended belief.
Under belief $\underline{\mu}=(0,1)$, state $t_1$ is believed to occur with probability zero, \ie,  $q=0$.
In this scenario, preference will be revised to $\text{AV@R}_{\frac{1}{3(1-q)}}=\text{AV@R}_{\frac{1}{3}}$ in state $t_2$, which is identical to the original preference.
Similarly, under belief $\Bar{\mu}=(1,0)$, only state $t_1$ have positive probability to occur, \ie, $q=1$.
Preference is revised to $\text{AV@R}_{\frac{1}{2}-\frac{1}{6q}}=\text{AV@R}_\frac{1}{3}$, which is also identical to the original assignment. 
Therefore, the sender, who prefers a lower value of the expectation of average risk, would aim to offset preference revision with his signals. 
This consequence is at odds with our model construction, which is designed to leverage the structure of preference revision.
Note, however, that although it seems the signaling adds no additional information to state realization, we emphasize that the signaling step takes place at the ex ante stage before the revealing of the states.
This indicates that, before the sender even has a chance to make an observation, the (unattainable) optimal signal rule assigns consistent beliefs to force the receiver to believe that only one state has positive probability to occur under each belief. 
The sender, who prefers lower risks, can benefit from persuasion as the prior belief splits, despite the fact that preferences stay the same on states that are consistent with beliefs.

Suppose now that the sender aims to increase the expectation of average risk.
On $0<q<1$, the maximum of average risk is $\frac{160}{3}$ at $q=\frac{1}{3}$.
Since this is the only local maximum point and the prior belief is identified by $q_0=\frac{1}{2}$, whether sender can benefit from persuasion depends on the property of average risk as a function of $q$ on $\frac{1}{3}\leq q\leq1$.
Let this function be denoted $f(q)$, \ie, $f(q)=\frac{-24q^2+148q+60}{1+3q}$.
Since $f(q)$ is convex on $[\frac{1}{3},1]$, we can split the prior to obtain a higher overall expectation of average risks. 
Consider the distribution of posterior beliefs $\eta$ such that $\text{supp}(\eta)=\{\mu^*, \Bar{\mu}\}$, where $\mu^*=(\frac{1}{3}, \frac{2}{3})$ and $\Bar{\mu}=(1,0)$. 
%\textcolor{blue}{******maybe we need graphical illustration*****}
For $\eta=(\gamma, 1-\gamma)$, Bayes plausibility indicates that $\gamma=\frac{3}{4}$.
The signal space can then be chosen again as contaning two elements, \ie, $S=\{s^*, \Bar{s}\}$, each recommending a corresponding belief.  
The signal rule then satisfies $\pi(s^*|t_1)=\pi(\Bar{s}|t_1)=\frac{1}{4}$, $\pi(s^*|t_2)=\frac{1}{2}$ and $\pi(\Bar{s}|t_2)=0$.
Similar as the previous case, this signal rule is not feasible and its corresponding value of expected average risk is not attainable due to the fact that $\Bar{\mu}$ is a boundary belief. 
We can then use the same method by substituting $\Bar{\mu}$ with a fully-supported belief $\Bar{\mu}'$ that is close and construct a signal rule to realize the new posteriors to obtain values of risks arbitrarily close to the upper bound.

Finally, we note that whether the sender prefers a lower or a higher value of the expectation of average risk depends on the alignment of the objectives of the sender and the receiver. 
In a scenario where the sender aims to persuade the receiver to protect a system, sender may aim to increase the overall level of risk aversion of the receiver so that she can adopt certain protective measures.
Only when the receiver perceives enough risk, can she be willing to perform the adoption.
However, in a scenario where the sender is malicious and he aims to attack the system, he will try to drop the receiver's guard by decreasing the overall level of risk aversion.
Consequently, the receiver may choose not to adopt any protection measure, leading to intrusion to the system.
We refer the reader to \citet{sayin2021bayesian} for more discussions on how persuasion can be related to security and deception.

\section{Proofs from Section \ref{sec:analysis:per se}.}
We first provide the proofs of results in Section \ref{sec:analysis:multiplicity}.

\proof{Proof of Lemma \ref{lemma:hemicontinuity of Z(mu)}.} 
Since the correspondence is defined by linear inequalities on a compact codomain, it has closed graph. 
This shows upper hemicontinuity.
We discuss lower hemicontinuity depending on the range of the level parameter $\alpha$.
When $\alpha=1$, the constraint $(1-\alpha)Z_t\leq \1$ collapses to $0\cdot Z_t \leq 1$. This means that there is no upper bound for $Z_t$.
Because $Z_t$ can scale up to infinity, the set does not experience a collapse when the probability assigned to some state $t\in T$ becomes zero and the equality constraint can be maintained.
So, lower hemicontinuity holds.
When $\alpha\in [0,1)$, lower hemicontinuity may fail at boundary distributions where a probability of a state transitions between a positive value and zero.
Because if a state's probability hits zero, its $Z_t$ parameter expands to the full interval $[0, 1/(1-\alpha)$.
However, due to the assumption that beliefs are fully supported, the expectation constraints always binds every coordinate of $Z$.
Then, the set does not experience sudden expansions and lower hemicontinuity holds.
Therefore, the correspondence is continuous.
%\Halmos
\endproof

\proof{Proof of Lemma \ref{lemma: v hat lower semi}.} 
Suppose that $\hat{v}$ is discontinuous at some $\mu\in\Delta(T)$. Observe that $\hat{v}(\mu)=\mathbb{E}_{\mu}[\text{AV@R}_{\alpha, \hat{Z}_t^*}(Y|t)]$ and $\text{AV@R}$ is continuous in its confidence level parameter. By Berge's theorem, the receiver must be indifferent between a set of risk adjustments at $\mu$, \ie, $\mathcal{Z}^*(\mu)$ is not a singleton. By definition, however, $\hat{v}(\mu)=\min_{\{Z_t^*\}\in \mathcal{Z}^*(\mu)} \mathbb{E}_{\mu}[\text{AV@R}_{\alpha, Z_t^*}(Y|t)]$. Hence, $\hat{v}$ is lower semicontinuous.%\Halmos
\endproof

In the following, we provide the proofs in Section \ref{sec:analysis increase average risk}.

\proof{Proof of Lemma \ref{lemma: W(Q, Q')<W(mu,mu')+ const}.} 
Let $\Phi=\mu\otimes P$ and $\Phi=\mu'\otimes P$ denote the joint probability measures constructed from $\mu$ and $\mu'$ with $P$.
We show the conclusion by using $W_1(\Phi, \Phi')$ as the intermediate quantity.
\\
First, observe that $W_1(Q,Q')\leq \int_{\Omega\times\Omega}\dd(\omega,\omega')\Lambda_{\Omega}^*(\dd \omega, \dd \omega')$, where $\Lambda_{\Omega}^*$ is constructed for measurable sets $A_\Omega, B_\Omega \in \Omega$ by
\begin{equation*}
    \Lambda_{\Omega}^*(A_\Omega\times B_\Omega)=\Lambda_{T\times\Omega}^*\big((T\times A_{\Omega})\times (T\times B_{\Omega})\big),
\end{equation*}
with $\Lambda_{T\times\Omega}^*\in\Delta\big((T\times\Omega)\times (T\times\Omega)\big)$ being the optimal transport plan associated with marginals $\Phi$ and $\Phi'$.
Then, we obtain that
\begin{equation*}
    \begin{aligned}
        W_1(Q,Q')\leq &\int_{(T\times\Omega)\times(T\times\Omega)}\dd(\omega, \omega')\Lambda_{T\times\Omega}^*(d t d \omega, d t' d \omega')
        \\
        \leq &
        \int_{(T\times\Omega)\times(T\times\Omega)}\dd \big((t,\omega),(t', \omega')\big)\Lambda_{T\times\Omega}^*(d t d \omega,d t' d \omega')
        \\
        = &
        W_1(\Phi, \Phi'),
    \end{aligned}
\end{equation*}
where the second inequality follows from the fact that $\dd(\cdot, \cdot)$ satisfies $\dd\big((t, \omega),(t', \omega')\big)\leq \dd(t, t')+\dd(\omega, \omega')$.
This shows that $W_1(Q,Q')\leq W_1(\Phi, \Phi')$.
\\
Second, for $\Gamma_{\mu,\mu'}^*$ that is the optimal transport plan associated with marginals $\mu$ and $\mu'$, let $\Psi_{t, t'}\in \Lambda_{\Omega\times\Omega}\big( P(\cdot|t), P(\cdot|t')\big)$ denote the joint probability measure on $\Omega\times \Omega$ with marginals $P(\cdot|t)$ and $P(\cdot|t')$.
Then, define for measurable sets $A_T\in T, B_T\in T$ and $A_\Omega\in \Omega, B_\Omega\in \Omega$ that
\begin{equation*}
    \Lambda_{T\times\Omega}^*\big((A_T\times A_\Omega)\times(B_T\times B_\Omega)\big) = \int_{A_T\times B_T} \Psi_{t,t'}(A_\Omega\times B_\Omega)\Gamma_{\mu,\mu'}^*(d t,d t').
\end{equation*}
This is a valid construction, since the marginals of $\Lambda_{T\times\Omega}$ can be correctly derived, \ie, 
\begin{equation*}
    \begin{aligned}
        \Lambda_{T\times\Omega}^*\big((A_T\times A_\Omega)\times(T\times \Omega)\big)
        =&
        \int_{A_T\times T}  \Psi_{t,t'}(A_\Omega\times \Omega)\Gamma_{\mu,\mu'}^*(d t,d t')
        \\
        =&
        \sum_{t\in A_T}P(A_\Omega|t)\mu(t)
        \\
        =& 
        \mu\otimes P\\
        =& \Phi.
    \end{aligned}
\end{equation*}
The other marginal can be recovered by a similar derivation.
Then, we obtain the following relation:
\begin{equation*}
    \begin{aligned}
        W_1(\Phi, \Phi')=&
        W_1(\mu\otimes P, \mu' \otimes P)
        \\
        \leq &
        \int_{(T\times\Omega)\times(T\times\Omega)}\dd \big((t,\omega),(t', \omega')\big)\Lambda_{T\times\Omega}^*(d t d \omega, d t' d \omega')
        \\
        = &
        \int_{(T\times\Omega)\times(T\times\Omega)}\dd \big((t,\omega),(t', \omega')\big)\Psi_{t,t'}(d \omega, d \omega')\Gamma_{\mu, \mu'}^*(d t, d t')
        \\
        \leq &
        \int_{(T\times\Omega)\times(T\times\Omega)}\Big(\dd(t,t')+\dd(\omega,\omega')\Big)\Psi_{t,t'}(d \omega, d \omega')\Gamma_{\mu, \mu'}^*(d t, d t')
        \\
        =&
        \underbrace{\int_{(T\times\Omega)\times(T\times\Omega)}\dd(t,t')\Psi_{t,t'}(d \omega, d \omega')\Gamma_{\mu, \mu'}^*(d t, d t')}_{:=E_1}
        \\
        & \quad +
        \underbrace{\int_{(T\times\Omega)\times(T\times\Omega)}\dd(\omega,\omega')\Psi_{t,t'}(d \omega, d \omega')\Gamma_{\mu, \mu'}^*(d t, d t')}_{:=E_2}
    \end{aligned}
\end{equation*}
The first term in the last equation above becomes $E_1=W_1(\mu, \mu')$ if we first integrate over $\Omega\times\Omega$.
The second term in the last equation above becomes
\begin{equation*}
    E_2=\int_{T\times T}\int_{\Omega\times\Omega}\dd(\omega,\omega')\Psi_{t,t'}(d \omega, d \omega')\Gamma_{\mu, \mu'}^*(d t, d t').
\end{equation*}
If $\Psi_{t, t'}$ is chosen optimally, then we have 
\begin{equation*}
    \begin{aligned}
        E_2=\int_{T\times T}W_1(P(\cdot|t),P(\cdot|t'))\Gamma_{\mu, \mu'}^*(d t, d t'),
    \end{aligned}
\end{equation*}
which is a constant since $P(\cdot|t)$ is fixed for all $t\in T$ and $\Gamma_{\mu, \mu'}^*$ is the optimal transport plan associated with marginals $\mu$ and $\mu'$.
Thus, we obtain $W_1(\Phi, \Phi')\leq W_1(\mu, \mu')+G(\mu, \mu')$.
\\
Combining the two parts above leads to the result in the lemma. 
%\Halmos
\endproof

\proof{Proof of Lemma \ref{lemma:v hat bound}.} 
Observe that
\begin{equation*}
    \begin{aligned}
        &\big|\mathbb{E}_{\mu}[\text{AV@R}_{\alpha, \hat{Z}_t^*(\mu)}(Y|t)]
    -
    \mathbb{E}_{\mu}[\text{AV@R}_{\alpha, \hat{Z}_t^*(\mu_0)}(Y|t)]\big|
    \\
    \leq &
    \underbrace{\big| \mathbb{E}_{\mu}[\text{AV@R}_{\alpha, \hat{Z}_t^*(\mu)}(Y|t)]
    -
    \mathbb{E}_{\mu}[\hat{Z}_t^*(\mu) \cdot \text{AV@R}_{\alpha, \hat{Z}_t^*(\mu)}(Y|t)] \big|
    }_{:=e_1}
    \\
    & \quad +
    \underbrace{\big| \mathbb{E}_{\mu}[\hat{Z}_t^*(\mu) \cdot \text{AV@R}_{\alpha, \hat{Z}_t^*(\mu)}(Y|t)]
    -
    \mathbb{E}_{\mu_0}[\hat{Z}_t^*(\mu_0) \cdot \text{AV@R}_{\alpha, \hat{Z}_t^*(\mu_0)}(Y|t)]
    \big| 
    }_{:=e_2}
    \\
    & \quad +
    \underbrace{\big| \mathbb{E}_{\mu}[\text{AV@R}_{\alpha, \hat{Z}_t^*(\mu_0)}(Y|t)]
    -
    \mathbb{E}_{\mu_0}[\hat{Z}_t^*(\mu_0) \cdot \text{AV@R}_{\alpha, \hat{Z}_t^*(\mu_0)}(Y|t)]
    \big|
    }_{:=e_3}.
    \end{aligned}
\end{equation*}
The term $e_3$ satisfies
\begin{equation*}
    \begin{aligned}
        e_3 
        = &\Big| \sum_{t} \text{AV@R}_{\alpha, \hat{Z}_t^*(\mu_0)}(Y|t)\left(\mu(t)-\mu_0(t)\hat{Z}_t^*(\mu_0) \right) \Big|
        \\ 
        \leq &
        \sum_{t} \Big|\text{AV@R}_{\alpha, \hat{Z}_t^*(\mu_0)}(Y|t)\Big|
        \cdot
        \Big|
        \mu(t)-\mu_0(t)\hat{Z}_t^*(\mu_0)
        \Big|
        \\
        = &
        \sum_{t} \Big|\text{AV@R}_{\alpha, \hat{Z}_t^*(\mu_0)}(Y|t)\Big|
        \cdot
        \Big|
        \mu(t)-\mu_0(t)+\mu_0(t)-\mu_0(t)\hat{Z}_t^*(\mu_0)
        \Big|
        \\
        = &
        \sum_{t} \Big|\text{AV@R}_{\alpha, \hat{Z}_t^*(\mu_0)}(Y|t)\Big|
        \cdot
        \left(
        \Big|
        \mu(t)-\mu_0(t)\Big|
        +\Big| \mu_0(t)(1-\hat{Z}_t^*(\mu_0)
        \Big|
        \right)
        \\
        \leq &
        M_0\cdot
        \left(
        \sum_{t}   
        \Big|
        \mu(t)-\mu_0(t)\Big|
        +\sum_{t}  \Big| \mu_0(t)(1-\hat{Z}_t^*(\mu_0)
        \Big|
        \right).
    \end{aligned}
\end{equation*}
Since $\mu, \mu_0 \in \Delta(T)$ and $0\leq \hat{Z}_t(\mu)\leq \frac{1}{1-\alpha}$ for all $\mu\in \Delta(T)$ and $t\in T$, we observe that 
\begin{equation*}
    \sum_t\big|\mu_0(t)(1-\hat{Z}_t^*(\mu_0)) \big|=\mathbb{E}_{\mu_0}[|1-\hat{Z}_t^*(\mu_0)|]\leq \max\{1, \frac{\alpha}{1-\alpha}\}.
\end{equation*}
This leads to $e_3\leq M_0(2\dd_{TV}(\mu, \mu_0)+m)$.
Using a derivation similar as above leads to $e_1\leq M'm$.
\\
Due to the decomposition of $\text{AV@R}$ in (\ref{eq:decomposition of avar}), we obtain that $e_2=|\text{AV@R}_{\alpha}^{\mu\circ P}(Y)-\text{AV@R}_{\alpha}^{\mu_0\circ P}(Y)|$. 
Hence, by combining Lemma \ref{lemma:e2 bound} with Lemma \ref{lemma: W(Q, Q')<W(mu,mu')+ const}, we have 
\begin{equation*}
    e_2\leq \frac{L(Y)}{1-\alpha}\cdot \Big(W_1(\mu, \mu_0)+G(\mu, \mu_0) \Big).
\end{equation*}
Therefore, we arrive at the conclusion in the lemma. 
%\Halmos
\endproof

\proof{Proof of Theorem \ref{thm:sender benefit no action}.} 
Since there is information that the sender would share, there exists a belief $\mu''\in\Delta(T)$ that satisfies condition (\ref{eq:information sender would share}), \ie, 
\begin{equation}
    \mathbb{E}_{\mu''}[\text{AV@R}_{\alpha, Z_t^*(\mu'')}(Y|t)]<\mathbb{E}_{\mu''}[\text{AV@R}_{\alpha, Z_t^*(\mu_0)}(Y|t)]-\epsilon.
    \label{eq:would share mu''}
\end{equation}
When risk preference persuasion is in place, the prior belief $\mu_0$ is split according to $\mu_0=\gamma\mu'+(1-\gamma)\mu''$ for some parameter $\gamma\in (0,1)$.
Then, the expected loss for the sender is $\inf_{\eta}\mathbb{E}_{\eta}[\hat{v}(\mu)]$ with the distribution of posterior beliefs $\eta$ being Bayes-plausible.
Let $\eta$ be such that $\text{supp}(\eta)=\{\mu', \mu''\}$ with $\eta(\mu')=\gamma$ and $\eta(\mu'')=1-\gamma$.
Then, by Bayes-plausibility, we have $\sum_{\text{supp}(\eta)}\mu\eta(\mu)=\mu_0$.
Then, to show that risk preference persuasion can decrease the average risk, it suffices to show that $\mathbb{E}_{\eta}[\hat{v}(\mu)]<\hat{v}(\mu_0)$ for $\eta$ constructed above.
Since $\hat{v}(\mu)=\mathbb{E}_{\mu}[\text{AV@R}_{\alpha, \hat{Z}_t^*(\mu)}(Y|t)]$, we can write 
\begin{equation*}
    \mathbb{E}_{\eta}[\hat{v}(\mu)]
    =\gamma\hat{v}(\mu')+(1-\gamma)\hat{v}(\mu'')
    =\gamma\mathbb{E}_{\mu'}[\text{AV@R}_{\alpha, \hat{Z}_t^*(\mu')}(Y|t)]+(1-\gamma)\mathbb{E}_{\mu''}[\text{AV@R}_{\alpha, \hat{Z}_t^*(\mu'')}(Y|t)],
\end{equation*}
and 
\begin{equation*}
    \hat{v}(\mu_0)
    =\mathbb{E}_{\mu_0}[\text{AV@R}_{\alpha, \hat{Z}_t^*(\mu_0)}(Y|t)]
    =\gamma\mathbb{E}_{\mu'}[\text{AV@R}_{\alpha, \hat{Z}_t^*(\mu_0)}(Y|t)]+(1-\gamma)\mathbb{E}_{\mu''}[\text{AV@R}_{\alpha, \hat{Z}_t^*(\mu'')}(Y|t)].
\end{equation*}
Then, applying Lemma \ref{lemma:v hat bound}  and condition (\ref{eq:would share mu''}) yields the conclusion. 
%\Halmos
\endproof

\section{Proofs from Section \ref{sec:analysis with action}.}

\proof{Proof of Theorem \ref{thm:existence of optimal value with action}.}
As the sender's problems (\ref{eq:persuasion with one action per belief}) and (\ref{eq:persuasion with one action per state}) differ only in whether the receiver's action depends on the state, the corresponding analysis also only differ in this aspect.
Thus, we focus on problem (\ref{eq:persuasion with one action per belief}).
With a slight abuse of notation, we denote the average loss of the sender as $\hat{v}(\mu)=\mathbb{E}_{\mu}[v(\hat{a}(\mu),t)]$, where $\hat{a}(\mu)$ is the action that minimizes the sender's objective in the set of the optimal solutions $a(\mu)$ to the receiver's problem in the first constraint of (\ref{eq:persuasion with one action per belief}).
Then, it suffices to show that $\hat{v}$ is lower semicontinuous.
Since $A$ is fixed, by Berge's theorem, it suffices to show that the receiver's optimization problem has a continuous objective, \ie, $\mathbb{E}_{\mu}[\text{AV@R}_{\alpha, Z_t^*(\mu)}(X_a|t)]$ is a continuous function, and $a(\mu)$ is an upper hemicontinuous correspondence.
Then, as AV@R is continuous in its confidence level parameter, we need to show that the optimal dual variable $Z_t^*(\mu)$ is continuous.
Remark 23 in \citet{pflug2016time} indicates that $Z_t^*$ for all $t\in T$ is unique if we have $Z^*$ is unique.
Consequently, given that (\ref{eq:decomposition of avar}) has a continuous objective, continuity of $Z_t^*(\mu)$ follows. 
This certifies the existence of optimal solution to (\ref{eq:persuasion with one action per belief}).
The same analysis procedure also applies to problem (\ref{eq:persuasion with one action per state}).
Therefore, we arrive at the conclusion in the theorem.
%\Halmos
\endproof

\section{Proofs from Section \ref{sec:application}.}

\proof{Proof of Lemma \ref{lemma:nu is increasing in q}.}
Let $H(\nu, q):=qe^{-t_1\nu}+(1-q)e^{-t_2\nu}-(1-\alpha)$.
Equation (\ref{eq:nu satisfy this eq}) indicates that $H(\nu, q)=0$.
Then, implicit function theorem tells us that $\nu$ as a function of $q$, satisfies
\begin{equation*}
    \frac{\dd \nu}{\dd q}=-\frac{\partial H/ \partial q}{\partial H/ \partial \nu}=-\frac{e^{-t_1\nu} -e^{-t_2\nu}}{-\big( qt_1e^{-t_1\nu}+(1-q)t_2 e^{-t_2\nu}\big)}.
\end{equation*}
Since $t_1<t_2$ by assumption, we have $\frac{\partial H}{\partial q}>0$ and $\frac{\partial H}{\partial \nu}<0$.
Thus, $\frac{\dd \nu}{\dd q}>0$.
This completes the proof. 
%\Halmos
\endproof

\proof{Proof of Lemma \ref{lemma:revised preferences under all alpha}.}
We proceed by deriving the revised confidence levels for the cases where $\alpha>\frac{\kappa}{1+\kappa}$ and $\alpha\leq\frac{\kappa}{1+\kappa}$ separately. Let $Y=X+h(I^*(X))-I^*(X)$.
\\
Suppose $\alpha>\frac{\kappa}{1+\kappa}$ holds.
Then, by (\ref{eq:optimal I of general insurance}), we obtain $Y=\min(X, d^*)+(1+\kappa)\mathbb{E}[(X-d^*)_+]$.
We denote $c=(1+\kappa)\mathbb{E}[(X-d^*)_+]$ as this quantity is a constant given $\alpha, \mu\circ P$, and $\kappa$.
If $X\leq d^*$, $Y=X+c$.
If $X>d^*$, $Y=d^*+c$.
Since $\esssup(Y)=d^*+c$, the upper tail of $Y$ is flat. 
Because of the definition of V@R, we have $\mathbb{P}(X>d^*)=1-\frac{\kappa}{1+\kappa}=\frac{1}{1+\kappa}$.
Thus, $\mathbb{P}(Y=d^*+c)=\frac{1}{1+\kappa}$.
Furthermore, $\alpha>\frac{\kappa}{1+\kappa}$ indicates $1-\alpha<\frac{1}{1+\kappa}$, meaning that the probability of the worst $1-\alpha$ tail is strictly smaller than the mass of the flat region, which has probability $\frac{1}{1+\kappa}$.
Consequently, $\text{V@R}_\alpha(Y)=d^*+c$ and $\text{AV@R}_\alpha(Y)=d^*+c$.
This means that the dual variable associated with $\text{AV@R}_\alpha(Y)$ is $\Tilde{Z}^*=\frac{1}{1-\alpha}\1_{\{x>d^*\}}$.
On the other hand, under the mixture distribution density $f_X(x)$, the ex ante risk of $X$ induced by belief $\mu$ is 
\begin{equation*}
    \text{AV@R}_\alpha(X)=\nu +\frac{1}{1-\alpha}\Big[ \frac{q}{t_1}e^{-t_1\nu}+\frac{1-q}{t_2}e^{-t_2\nu} \Big],
\end{equation*}
whose associated dual variable is $Z^*=\frac{1}{1-\alpha}\1_{\{x\geq \nu\}}$ with $\nu$ satisfying (\ref{eq:nu satisfy this eq}).
Observe that $d^*$ is equivalent to solving (\ref{eq:nu satisfy this eq}) for $\nu$ with the right-hand side of (\ref{eq:nu satisfy this eq}) modified to $\frac{1}{1+\kappa}$ and we have that $\alpha>\frac{\kappa}{1+\kappa}$, it holds that $\nu>d^*$.
Therefore, the right choice of dual variable is $Z^*=\frac{1}{1-\alpha}\1_{\{x\geq \nu\}}$, thereby $Z_t^*=\mathbb{E}[Z^*|t]=\frac{1}{1-\alpha}\mathbb{P}(X\geq \nu|t)=\frac{e^{-t\nu}}{1-\alpha}$.
Then, $\alpha_t=1-(1-\alpha)Z_t^*=1-e^{-t\nu}$ for $t\in T$.
\\
Suppose that $\alpha\leq\frac{\kappa}{1+\kappa}$  holds. 
Then $\text{AV@R}_\alpha(Y)=\text{AV@R}_\alpha(X)$ under the optimal indemnity.
Similar as above, we obtain the optimal dual variable $Z^*=\frac{1}{1-\alpha}\1_{\{x\geq \nu\}}$ with $\nu$ satisfying (\ref{eq:nu satisfy this eq}).
However, since in this case $\alpha\leq \frac{\kappa}{1+\kappa}$, we have $\nu\leq d^*$.
Therefore, $Z_t^*=\mathbb{E}[Z^*|t]=\frac{e^{-t\nu}}{1-\alpha}$.
This will lead to the same expression for $\alpha_t$ and the proof is complete. 
%\Halmos
\endproof

\proof{Proof of Lemma \ref{lemma:q underline < q bar}.}
Let $C=1-\alpha$. Then, we have $0<C<1$.
Let $z=\frac{1}{1+\kappa}$. Then, we have $0<z<1$.
Let $r=t_1/t_2$. Then,  we have $0<r<1$.
With these auxiliary terms, we obtain $\underline{q}=\frac{C-z}{z^r-z}$ and $\Bar{q}=\frac{C-z^{1/r}}{z-z^{1/r}}$.
Their different is 
\begin{equation*}
    \underline{q}-\Bar{q}=\frac{(C-z)(z-z^{1/r})-(C-z^{1/r})(z^r-z)}{(z^r-z)(z-z^{1/r})}.
\end{equation*}
Since $0<z<1$ and $0<r<1<1/r$, we have $z^r-z>0$ and $z-z^{1/r}>0$.
So, the sign of $\underline{q}-\Bar{q}$ only depends on its numerator. 
Denote by $\delta$ the numerator of $\underline{q}-\Bar{q}$.
Algebraic manipulation leads to $\delta=\delta(C)=C(2z-z^r-z^{1/r})-(z^2-z^{r+1/r})$, thereby $\delta(C)$ is linear function for $0<C<1$.
Observe that $\delta(0)=-(z^2-z^{r+1/r})$.
As $r+\frac{1}{r}>2\cdot \sqrt{r\cdot \frac{1}{r}}=2$ and $0<x<1$, we have $\delta(0)<0$.
On the other hand, $\delta(1)=(1-z)(z-z^{1/r})-(1-z^{1/r})(z^r-z)$.
We identify $\delta(1)=h(z, r)$ to investigate its sign.
For all $z$, we have $h(z,1)=(1-z)(z-z)-(1-z)(z-z)=0$.
By calculation, $\delta(1)=2z-z^2-z^{1/r}-z^r+z^{r+1/r}$, whose partial derivative with respect to $r$ is
\begin{equation}
\begin{aligned}
    \frac{\partial \delta(1)}{\partial r}&=ln(z)\Big[  \frac{1}{r^2}z^{\frac{1}{r}}-z^r+(1-\frac{1}{r^2})z^{r+\frac{1}{r}}\Big]
    \\
    &=ln(z)\cdot z^r\cdot 
    \Big[ \frac{1}{r^2}z^{\frac{1}{r}-r}-1+(1-\frac{1}{r^2})z^{\frac{1}{r}}\Big]
    \\
    &=:ln(z)\cdot z^r \cdot g(z).
    \label{eq: d delta d r}
\end{aligned}
\end{equation}
If we can show that $\frac{\partial \delta(1)}{\partial r}>0$, then as $h(z,1)=0$, we can obtain $h(z,r)<0$ for $0<r<1$.
As $ln(z)<0$ and $z^r>0$ for $0<z<1$ and $0<r<1$, it suffices to show that $g(z)$ defined in the last equation in (\ref{eq: d delta d r}) is negative on $0<x<1$.
Observe that $\lim_{x\rightarrow  0^+}g(z)=-1$ and $\lim_{x\rightarrow 1^-}=0$.
Then, it suffices to show that $g(z)$ is increasing on $0<z<1$.
Algebraic calculation leads to \begin{equation*}
    g'(z)=\frac{1}{r}z^{\frac{1}{r}-r-1}\cdot(1-\frac{1}{r^2})(z^r-1)>0,
\end{equation*}
which follows since $0<r<1$ and $0<z<1$.
Thus, we have $\delta(1)<0$.
As $\delta(C)$ is linear, $\delta(0)<0$, and $\delta(1)<0$, we conclude that $\underline{q}<\Bar{q}$.
%\Halmos
\endproof

\proof{Proof of Theorem \ref{thm:conditions for benefiting reinsurer}.}
A necessary condition for the reinsurer to benefit is that, under some state, the revised preference exceeds the threshold $\frac{\kappa}{1+\kappa}$ while the original preference under that state falls below it.
In the sequel, we use this necessary condition to identify reduced regions of parameters and then compare $\mathbb{E}_\eta v(\mu)$ with $v(\mu_0)$ in each of the regions.
\\
Case (1).
In this case, the reinsurer's expected loss under the prior belief is $v(\mu_0)=-\frac{\kappa}{1+\kappa}(\frac{q_0}{t_1}+\frac{1-q_0}{t_2})$.
On the other hand, we aim to minimize $\mathbb{E}_\eta v(\mu)$ by choosing $\eta$ supported on $\{\mu', \mu''\}$ with probabilities $(\gamma, 1-\gamma)$.
By Bayes-plausibility, we have $q_0=\gamma q' +(1-\gamma)q''$, which leads to 
\begin{equation*}
    \begin{aligned}
        \mathbb{E}_\eta v(\mu)
        &= \gamma v(\mu')+(1-\gamma)v(\mu'')\\
        &= \gamma \Big[\frac{-\kappa}{1+\kappa}\Big(\frac{q'}{t_1}+\frac{1-q'}{t_2}\Big)\Big]
        +
        (1-\gamma) \Big[\frac{-\kappa}{1+\kappa}\Big(\frac{q''}{t_1}+\frac{1-q''}{t_2}\Big)\Big]
        \\
        &=
        \frac{-\kappa}{1+\kappa}\Big( \frac{q_0}{t_1}+\frac{1-q_0}{t_2} \Big).
    \end{aligned}
\end{equation*}
Thus, due to linearity of $v(\cdot)$, if the two beliefs $\mu'$ and $\mu''$ both leads to the positive loss scenario under both states, \ie, the revised confidence levels under both beliefs exceed the threshold $\frac{\kappa}{1+\kappa}$ in both states, then the sender cannot strictly benefit.
This indicates that the optimal average loss that the reinsurer can obtain by choosing a distribution of posterior beliefs $\eta$ is equal to the loss that he receives without performing persuasion. 
In fact, the reinsurer can be worse-off if he performs signaling due to the following reasons.
From Lemma \ref{lemma:nu is increasing in q} and the fact that a revised preference  $\alpha_t$ is strictly increasing in both $\nu$ and $t$, we observe by $q'<q_0<q''$ that the revised preference $\alpha'_{t_1}$ induced by belief $\mu'$ in state $t_1$ is the smallest among all of the four preferences induced by $\mu'$ and $\mu''$ at the two states. 
Thus, by (\ref{eq:value is 0 or + based on alpha}), we have to guarantee that $\alpha'_{t_1}>\frac{\kappa}{1+\kappa}$ in order to prevent its corresponding loss to become $0$.
Then, the condition $\alpha'_{t_1}>\frac{\kappa}{1+\kappa}$ leads to $\nu\leq \frac{ln(1+\kappa)}{t_1}$ with $\nu$ being the solution to $q'e^{-t_1\nu}+(1-q')e^{-t_2\nu}=1-\alpha$.
By Lemma \ref{lemma:nu is increasing in q}, $q'$ is upper-bounded as $\nu$ is also upper-bounded and the bound is given by choosing $\nu=\frac{ln(1+\kappa)}{t_1}$.
Then, we arrive at $q' \geq \Bar{q}$.
\\
Case (2).
In this case, the reinsurer's expected loss under the prior belief is $v(\mu_0)=\frac{-\kappa(1-q_0)}{t_2(1+\kappa)}$.
Due to the discussions in Case (1), the minimum necessary condition for the reinsurer to benefit from persuasion is that the revised preference induced by the higher belief $\mu''$ under state $t_1$ exceeds the threshold, \ie,  $\alpha_{t_1}''>\frac{\kappa}{1+\kappa}$.
This makes the corresponding cost jumps from $0$ to $\frac{1}{t_1(1+\kappa)}>0$.
Thus, following a similar procedure as in Case (1), we must have that $q''>\Bar{q}$.
However, this is not sufficient, as the lower belief $\mu'$ could induce a revised preference that is too low even under state $t_2$, \ie, $\alpha_{t_2}'<\frac{\kappa}{1+\kappa}$, such that the original positive loss under $t_2$ could jump to $0$.
Consequently, if $\alpha_{t_1}''>\frac{\kappa}{1+\kappa}$ and $\alpha_{t_2}'>\frac{\kappa}{1+\kappa}$, the reinsurer strictly benefit from persuasion.
Note that by Lemma \ref{lemma:q underline < q bar}, when $q''>\Bar{q}$ is satisfied, it is required that the condition on $q'$ holds simultaneously with the condition on $q''$  to enable the corresponding requirements on the revised preferences $\alpha_{t_2}'$ and $\alpha_{t_1}''$.
\\
If, on the other hand, we have $\alpha_{t_1}''>\frac{\kappa}{1+\kappa}$ and $\alpha_{t_2}'<\frac{\kappa}{1+\kappa}$, we need to compare losses explicitly. 
Now, the beliefs induce expected losses $v(\mu')=0$ and $v(\mu'')=\frac{-\kappa}{1+\kappa}\Big(\frac{q''}{t_1}+\frac{1-q''}{t_2}\Big)$.
Then, 
\begin{equation*}
\begin{aligned}
    \mathbb{E}_\eta v(\mu)&=\frac{-\kappa}{1+\kappa}(1-\gamma)\Big(\frac{q''}{t_1}+\frac{1-q''}{t_2}\Big)
    \\
    &=\frac{-\kappa}{1+\kappa}\cdot
    \underbrace{\Big( \frac{q_0-q'}{q''-q'}\Big)\cdot\Big(\frac{q''}{t_1}+\frac{1-q''}{t_2}\Big)}_{:=g(q',q'')},
\end{aligned}
\end{equation*}
where the last equality is a consequence of Bayes-plausibility.
So, minimizing $\mathbb{E}_\eta v(\mu)$ is equivalent to maximizing $g(q',q'')$ under the constraints of this case.
The partial derivative of $g$ satisfies $\partial g(q',q'')/ \partial q'<0$.
Hence, to maximize $g$, we pick $q'=0$.
Then, we can verify that $g(0, q'')$ is decreasing in $q''$, leading to picking $q''=\Bar{q}$.
This yields $\gamma=1-q_0/\Bar{q}$ and the expected loss induced by $\eta$ is $\mathbb{E}_{\eta}v(\mu)=\frac{q_0}{t_1}+\frac{q_0(1-\Bar{q})}{\Bar{q}t_2}$.
This loss value is strictly greater than $v(\mu_0)=\frac{-\kappa (1-q_0)}{t_2(1+\kappa)}$ under the assumption that $q_0<\Tilde{q}$.
Therefore, the reinsurer cannot strictly benefit if $\alpha_{t_1}''>\frac{\kappa}{1+\kappa}$ and $\alpha_{t_2}'<\frac{\kappa}{1+\kappa}$.
This concludes the analysis in Case (2).
\\
Case (3). 
In this case, the reinsurer's expected loss under the prior belief is $v(\mu_0)=0$.
From previous discussions, we know that $\alpha_{t_2}''$ is the largest among all of the four revised preferences induced by beliefs $\mu'$ and $\mu''$ under the two states.
Therefore, as long as it exceeds the threshold, the reinsurer strictly benefits from persuasion.
The condition for this to happen is $\alpha_{t_2}''>\frac{\kappa}{1+\kappa}$, which is equivalent to requiring $q''>\underline{q}$.
%\Halmos
\endproof

\proof{Proof of Theorem \ref{thm:optimal value in beneficial cases}.}
Case (2).
In this case, the reinsurer only strictly benefit if $q''>\Bar{q}$ and $q'>\underline{q}$ hold simultaneously.
Accordingly, expected losses are
\begin{equation*}
    v(\mu')=\frac{-\kappa(1-q')}{t_2(1+\kappa)},
    \text{ and }v(\mu'')=\frac{-\kappa}{1+\kappa}\Big[ \frac{q''}{t_1}+\frac{1-q''}{t_2} \Big].
\end{equation*}
Then, distribution of beliefs $\eta$ induces
\begin{equation*}
    \mathbb{E}_{\eta}v(\mu)
    =\frac{-\kappa}{1+\kappa}\Big[ \gamma \frac{1-q'}{t_2}+(1-\gamma)\Big(\frac{q''}{t_1}+\frac{1-q''}{t_2}\Big) \Big].
\end{equation*}
Using Bayes-plausibility, \ie, $\gamma q'+(1-\gamma)q''=q_0$, the above expression becomes
\begin{equation*}
    \mathbb{E}_{\eta}v(\mu)
    =\frac{-\kappa}{1+\kappa}\Big[ \frac{1-q_0}{t_2}+(1-\gamma)\frac{q''}{t_1} \Big].
\end{equation*}
Then, minimizing  $\mathbb{E}_{\eta}v(\mu)$ is equivalent to maximizing $(1-\gamma)q''$, under the constraints that $q''>\Bar{q}$, $q'>\underline{q}$, and Bayes-plausibility. 
To arrive at the optimal choice, we pick $q'=\underline{q}$ and obtain $q''=\frac{q_0-\gamma\underline{q}}{1-\gamma}$.
Since $q''>\Bar{q}$, we have $\frac{q_0-\gamma \underline{q}}{1-\gamma}>\Bar{q}$.
By Lemma \ref{lemma:q underline < q bar}, we obtain $\gamma>\frac{\Bar{q}-q_0}{\Bar{q}-\underline{q}}$.
Then, if $\Bar{q}>q_0$, the optimal choice is $\gamma=\frac{\Bar{q}-q_0}{\Bar{q}-\underline{q}}$.
Indeed, since Case (2) requires that $\alpha_{t_1}<\frac{\kappa}{1+\kappa}$, if we solve for $q_0$ using an analogue of (\ref{eq:nu satisfy this eq}) with $q_0$ at the boundary the boundary of this condition, then we would obtain $q_0=\Bar{q}$.
Consequently, it always holds in this case that $\Bar{q}>q_0$.
Therefore, optimal choices to minimize $\mathbb{E}_\eta v(\mu)$ are $q'=\underline{q}$ and $\gamma=\frac{\Bar{q}-q_0}{\Bar{q}-\underline{q}}$, and correspondingly, $q''=\frac{q_0-\gamma \underline{q}}{1-\gamma}$.
The signal rule that induces the optimal loss can be reconstructed using (\ref{eq:optimal signal rule}), which, following the procedure illustrated in Section \ref{sec:example}, is obtained as stated in the theorem.
\\
Case (3).
In this case, under the assumption that $t_2>t_1(1/\underline{q}-1/\Bar{q})$, the reinsurer benefits the most if $v(\mu'')=\frac{-\kappa}{1+\kappa}\Big( \frac{q''}{t_1}+\frac{1-q''}{t_2} \Big)$.
Note that this is compared with the scenario where the loss induced by $\alpha_{t_1}''$ is $0$, \ie, the $q''/t_1$ element in $v(\mu'')$ becomes $0$.
To obtain $v(\mu'')$, the requirement on $q''$ is $q''>\Bar{q}$, which automatically leads to $q''>\underline{q}$.
Then, according to the proof of Theorem \ref{thm:conditions for benefiting reinsurer}, maximizing $\mathbb{E}_\eta v(\mu)$ is equivalent to maximizing $g(q',q'')$.
Since we require $q''>\Bar{q}$, the optimal choice is $q'=0$ and $q''=\Bar{q}$, which leads to $\gamma=1-q_0/\Bar{q}$.
This leads to the optimal average loss in the theorem and the signal rule can be constructed using (\ref{eq:optimal signal rule}).
%\Halmos
\endproof

\end{appendices}

%\end{APPENDIX}

%Bibliography
\bibliographystyle{abbrvnat}  
\bibliography{references}  

\nocite{*}

\end{document}